\documentclass[12pt,a4paper,reqno]{amsart}
\usepackage{amsmath,amsfonts,amsthm,amssymb,color}
\usepackage{cmmib57}
\usepackage{exscale}
\usepackage{amscd}
\usepackage{latexsym}
\usepackage{graphicx}
\usepackage[T1]{fontenc}
\usepackage[latin1]{inputenc}
\definecolor{rp}{rgb}{0.25, 0, 0.75}
\definecolor{dg}{rgb}{0, 0.5, 0}

\newcommand{\der}{\delta}
\newcommand{\id}{\mbox{Id}}

\newcommand{\ist}{\int_{s}^{t}}
\newcommand{\norm}[1]{\lVert #1\rVert}

\newcommand{\ott}{[0,T]}

\newcommand{\cb}{{\mathcal B}}
\newcommand{\cac}{{\mathcal C}}

\newcommand{\cf}{{\mathcal F}}

\newcommand{\ch}{{\mathcal H}}
\newcommand{\cj}{{\mathcal J}}

\newcommand{\cn}{{\mathcal N}}

\newcommand{\cz}{{\mathcal Z}}

\newcommand{\al}{\alpha}

\newcommand{\ga}{\gamma}
\newcommand{\ep}{\varepsilon}

\newcommand{\si}{\sigma}

\newcommand{\laa}{\Lambda}

\newcommand{\R}{{\mathbb R}}

\newcommand{\lcl}{\left\{}
\newcommand{\rcl}{\right\}}
\newcommand{\lp}{\left(}
\newcommand{\rp}{\right)}
\newcommand{\lc}{\left[}
\newcommand{\rc}{\right]}

\newtheorem{thm}{Theorem}[section]
\newtheorem{cor}[thm]{Corollary}
\newtheorem{lem}[thm]{Lemma}
\newtheorem{prop}[thm]{Proposition}
\newtheorem{defn}{Definition}[section]
\newtheorem{hyp}{Hypothesis}[section]

\theoremstyle{remark}
\newtheorem{rem}{Remark}[section]

\newcommand{\Ind}{\;{\rm l}\hskip -0.23truecm 1}
\begin{document}

\title[Fractional differential systems]{Some differential systems driven by a fBm with Hurst parameter greater than 1/4}

\author{Samy Tindel \and Iván Torrecilla}

\address{
{\it Samy Tindel:} {\rm Institut {\'E}lie Cartan Nancy, B.P. 239,
54506 Vand{\oe}uvre-l{\`e}s-Nancy Cedex, France}. {\it Email:}{\tt
tindel@iecn.u-nancy.fr}
\newline
$\mbox{ }$\hspace{0.1cm} {\it Iván Torrecilla:} {\rm Facultat de
Matem\`atiques, Universitat de Barcelona, Gran Via 585, 08007
Barcelona, Spain}. {\it Email: }{\tt
itorrecillatarantino@gmail.com}. {\it Supported by the grants
HF2005-0038, MTM2006-01351 from the Direcci\'on General de
Investigaci\'on, Ministerio de Educación y Ciencia, Spain.} }

\begin{abstract}
This note is devoted to show how to push forward the algebraic integration setting in order to treat differential systems driven by a noisy input with Hölder regularity greater than $1/4$. After recalling how to treat the case of ordinary stochastic differential equations, we mainly focus on the case of delay equations. A careful analysis is then performed in order to show that a fractional Brownian motion  with Hurst parameter $H>1/4$ fulfills the assumptions of our abstract theorems.
\end{abstract}

\keywords{Rough paths theory; Stochastic delay equations; Fractional Brownian motion.}

\subjclass[2000]{60H05, 60H07, 60G15}

\date{\today}
\maketitle

\section{Introduction}
A differential equation driven by a $d$-dimensional fractional Brownian motion $B=(B^1,\ldots,$ $B^d)$ is generically written as:
\begin{equation}
\label{eq:sde-intro}
y_t=a+ \int_0^t \sigma(y_s) \, dB_s,\quad t\in[0,T],
\end{equation}
where $a$ is an initial condition in $\R^n$, $\si:\R^n\to\R^{n,d}$ is a smooth enough function, and $T$ is an arbitrary positive constant. The recent developments in rough paths analysis \cite{CQ02,LyQi02,FV08} have allowed to solve this kind of differential equation when the Hurst parameter $H$ of the fractional Brownian motion is greater than 1/4, by first giving a natural meaning to the integral  $\int_0^t \sigma(y_s) \, dB_s$ above. It should also be stressed that a great amount of information has been obtained about these systems, ranging from support theorems \cite{Fr05} to the existence of a density for the law of $y_t$ at a fixed instant $t$ (see \cite{CF07,CFV07}).

\smallskip

In a parallel but somewhat different direction, the algebraic
integration theory (introduced in \cite{Gu04}), is meant as an
alternative and complementary method of generalized integration
with respect to a rough path. It relies on some more elementary
and explicit formulae, and its main advantage is that it allows to
develop rather easily an intuition about the way to handle
differential systems beyond the diffusion case given by
(\ref{eq:sde-intro}). This fact is illustrated by the study of
delay \cite{NNT07} and Volterra \cite{DT08} type equations, as
well as an attempt to handle partial differential equations driven
by a rough path \cite{GT08}. In each of those cases, the main
underlying idea consists in changing slightly the basic structures
allowing a generalized integration theory (discrete differential
operator $\delta$, sewing map $\Lambda$, controlled processes) in
order to adapt them to the context under consideration. While the
technical details might be long and tedious, let us insist on the
fact that the changes in the structures we have alluded to are
always natural and (almost) straightforward. Some twisted L\'evy
areas also enter into the game in a natural manner.

\smallskip

However, all the results contained in the references mentioned
above concern a fractional Brownian motion $B$ with Hurst
parameter $H>1/3$, while the usual rough path theory enables to
handle any $H>1/4$  (see \cite{CQ02} for the explicit application
to fBm). The current paper can then be seen as a step in order to
fill this gap, and we shall deal mainly with two kind of systems:
first of all, we will show how to solve equation
(\ref{eq:sde-intro}) when $1/4<H\leq 1/3$, thanks to the algebraic
integration theory. The results we will obtain are not new, and
the algebraic integration formalism has been extended to a much
broader context in \cite{Gu06} by means of a tree-based expansion
(let us mention again that the case $H>1/4$ is also covered by the
usual rough path theory). This study is thus included here as a
preliminary step, where the changes in the structures (new
definition of a controlled path, introduction of a L\'evy
\emph{volume}) can be  exhibited in a simple enough manner.

\smallskip

Then, in a second part of the paper, we show how to adapt our formalism in order to deal with delay equations of the form:
\begin{equation}
\label{eq:intro-def-delay-eq} \left\{
\begin{array}{ll}
dy_t=\sigma(y_t,y_{t-r_1},\ldots,y_{t-r_{q}}) \, dB_t\quad  t\in[0,T],    \\
y_t=\xi_t,  \qquad  t \in [-r_{q},0],
\end{array}
\right.
\end{equation}
where $y$ is a $\R^n$-valued continuous process, $q$ is a positive
integer, $\si:\R^{n,q+1}\to\R^{n,d}$ is a smooth enough function,
$B$ is a $d$-dimensional fractional Brownian motion with Hurst
parameter $H>1/4$ and $T$ is an arbitrary positive constant. The
delay in our equation is represented by the family
$0<r_1<\ldots<r_{q}<\infty$, and the initial condition $\xi$ is
taken as a regular enough deterministic function on $[-r_{q},0]$.
Though this kind of system is implicitly considered in \cite{Ho06}
in the usual Brownian case, and in \cite{FR06} for a Hurst
parameter $H>1/2$, the rough paths techniques have only been used
in this context (to the best of our knowledge) in \cite{NNT07},
where a delay equation driven by a fractional Brownian motion with
Hurst parameter $H>1/3$ is considered. Our paper is thus an
extension of this last result, and we shall obtain an existence
and uniqueness theorem for equation (\ref{eq:intro-def-delay-eq})
in the case $H>1/4$, under reasonable regularity conditions on
$\si$ and $\xi$.

\smallskip

From our point of view the example of delay equations, which is interesting in its own right because of its potential physical applications, is also worth studying in order to see the kind of algebraical structures which pop out when changing the type of rough differential system we are trying to handle. In case of a delay equation driven by a rough path of order 3 like ours, we shall introduce the notion of \emph{doubly delayed} controlled processes, and have to assume a priori the existence of some \emph{doubly delayed} elements of area and volume associated to $B$. This rich structure induces some cumbersome computations when one decides to expand all the calculations explicitly like we did. However, in the end, one also gets the satisfaction to see that the algebraic integration setting is flexible enough to be adapted naturally to many situations. Let us also mention that the infinite dimensional setting of \cite{Mo84} is avoided here, and that all our considerations only involve paths taking values in a finite dimensional space.

\smallskip

Let us also mention that, as in other examples of fractional differential systems, an important part of our work consists in verifying that the fractional Brownian motion satisfies the assumptions of our abstract theorems. The main available tools we are aware of for this kind of task are based on Russo-Vallois approximations \cite{RV93}, analytic approximations of the fBm (like we did in \cite{DT08}) or Malliavin calculus. We have chosen here to work under this latter framework, since it leads to reasonably short calculations, and also because it allows us to build on the previous results obtained in \cite{NNT07}, where this formalism was also adopted.

\smallskip

Here is how our article is structured: Section \ref{s2} is devoted to recall the basic ingredients of the algebraic integration setting. The diffusion case is treated at Section \ref{sec:diff-case}, and the bulk of the computations concerning delay systems can be found at Section \ref{sec:delay}. Finally, the application to fractional Brownian motion is given at Section \ref{sec:fbm}.

\section{Increments}
\label{s2}

To  begin with, let us present the very basic algebraic structures which
will allow to define a pathwise integral with respect to
irregular functions.

\subsection{Basic notions of algebraic integration}\label{sec:basics-alg-integration}

For an arbitrary real number
$T>0$, a vector space $V$ and an integer $k\ge 1$ we denote by
$\cac_k(V)$ the set of functions $g : [0,T]^{k} \to V$, $g(t_1,\ldots,t_k)= g_{t_1 \ldots t_{k}}$ such
that $g_{t_1 \cdots t_{k}} = 0$ whenever $t_i = t_{i+1}$ for some $1\le i\le k-1$.
Such a function will be called a
\emph{$(k-1)$-increment}, and we will
set $\cac_*(V)=\cup_{k\ge 1}\cac_k(V)$.

\smallskip

On $\cac_k(V)$ we introduce the operator $\der$ defined as follows:
\begin{equation}
  \label{f2.1}
\delta : \cac_k(V) \to \cac_{k+1}(V), \qquad
(\delta g)_{t_1 \cdots t_{k+1}} = \sum_{i=1}^{k+1} (-1)^{k-i}
g_{t_1  \cdots \hat t_i \cdots t_{k+1}} ,
\end{equation}
where $\hat t_i$ means that this particular argument is omitted.
A fundamental property of $\der$, which is easily verified,
is that
$\delta\circ \delta = 0$. We will denote $\cz\cac_k(V) = \cac_k(V) \cap \text{Ker}\delta$
and $\cb \cac_k(V) =
\cac_k(V) \cap \text{Im}\delta$.

\smallskip

Throughout the paper we will mainly deal with actions of $\delta$ on $\cac_i$, $i=1,2$. That is, consider
$g\in\cac_1$ and $h\in\cac_2$. Then, for any $s,u,t\in\ott$, we have
\begin{equation}
\label{f2.2}
  (\der g)_{st} = g_t - g_s,
\quad\mbox{ and }\quad
(\der h)_{sut} = h_{st}-h_{su}-h_{ut}.
\end{equation}
Furthermore, it is easily checked that
$\cz \cac_{k+1}(V) = \cb \cac_{k}(V)$ for any $k\ge 1$.
In particular, we have the following property:
\begin{lem}
\label{l2.1}Let $k\ge 1$ and $h\in \cz\cac_{k+1}(V)$. There exists a (non unique)
$f\in\cac_{k}(V)$ such that $h=\der f$.
\end{lem}

Lemma \ref{l2.1} implies that all the elements
$h \in\cac_2(V)$ such that $\der h= 0$ can be written as $h = \der f$
for some (non unique) $f \in \cac_1(V)$. Thus we have a heuristic
interpretation of $\der |_{\cac_2(V)}$ as a measure of how much a
given 1-increment  is far from being an  exact increment of a
function, i.e., a finite difference.
\begin{rem}
\label{r2.1}Here is a first elementary but important link between these algebraic structures
and integration theory. Let $f$ and $g$ be two smooth real valued functions on $\ott$.
Define $I\in\cac_2$ by
$$
I_{st}=\ist \lp\int_{s}^{v} dg_w\rp df_v ,
\quad\mbox{ for }\quad
s,t\in\ott.
$$
Then,
$(\der I)_{sut}=[g_u-g_s] [f_t-f_u]=(\der g)_{su} (\der f)_{ut}$.
Hence we see that the operator $\der$ transforms iterated integrals
into products of increments, and we will be able to take advantage of both regularities
of $f$ and $g$ in these products of the form $\der g\, \der f$.
\end{rem}

\vspace{0.3cm}

Let us concentrate now on the case $V=\R^d$, and
notice that our future discussions will mainly rely on
$k$-increments with $k \le 2$, for which we will use some
analytical assumptions. Namely,
we measure the size of these increments by H\"older-type norms
defined in the following way. For $f \in \cac_2(V)$ and $\mu\in(0,\infty)$, let
\begin{equation}
\label{eq:2.4}
\norm{f}_{\mu} =
\sup_{s,t\in\ott}\frac{|f_{st}|}{|t-s|^\mu},
\end{equation}
and set $\cac_2^\mu(V)=\lcl f \in \cac_2(V);\, \norm{f}_{\mu}<\infty  \rcl$.

The usual H\"older spaces $\cac_1^\mu(V)$ will be determined
in the following way. For a continuous function $g\in\cac_1(V)$, we simply set
\begin{equation}\label{hnorm}
\|g\|_{\mu}=\|\der g\|_{\mu},
\end{equation}
where the right-hand side of this equality is defined after (\ref{eq:2.4});
we will say that $g\in\cac_1^\mu(V)$ iff $\|g\|_{\mu}$ is finite.
Notice that $\|\cdot\|_{\mu}$ is only a semi-norm on $\cac_1(V)$.
However we will generally work on spaces of the type
\begin{equation*}%\label{def:hold-init}
\cac_{1,a}^\mu(V)=
\lcl g:\ott\to V;\, g_0=a,\, \|g\|_{\mu}<\infty \rcl,
\end{equation*}
for a given $a\in V$, on which $\|g\|_{\mu}$ then becomes a norm.
 For $h \in \cac_3(V)$ we set %in the same way
\begin{eqnarray*}
 % \label{eq:normOCC2}
  \norm{h}_{\gamma,\rho} &=& \sup_{s,u,t\in\ott}
\frac{|h_{sut}|}{|u-s|^\gamma |t-u|^\rho}\\
%\quad\mbox{and}\quad
\|h\|_\mu &= &
\inf\left \{\sum_i \|h_i\|_{\rho_i,\mu-\rho_i} ;\, h =
 \sum_i h_i,\, 0 < \rho_i < \mu \right\} ,\nonumber
\end{eqnarray*}
where the last infimum is taken over all sequences $\{h_i \in \cac_3(V) \}$
such that $h
= \sum_i h_i$. % and for all choices of the numbers $\rho_i \in (0,z)$.
Then  $\|\cdot\|_\mu$ is easily seen to be a norm on $\cac_3(V)$, and we set
$$
\cac_3^\mu(V):=\lcl h\in\cac_3(V);\, \|h\|_\mu<\infty \rcl.
$$
Eventually,
let $\cac_j^{1+}(V) = \cup_{\mu > 1} \cac_j^\mu(V)$, $j=1,2,3$,
and remark that the same kind of norms can be considered on the
spaces $\cz \cac_3(V)$, leading to the definition of some spaces
$\cz \cac_3^\mu(V)$ and $\cz \cac_3^{1+}(V)$.
\vspace{0.3cm}

With these notations in mind,
the crucial point in our approach to pathwise integration of irregular
processes is that, under mild smoothness conditions, the operator
$\delta$ can be inverted. This inverse is called $\laa$, and is
defined in the following  proposition, whose proof can be found in \cite{Gu04,GT08}:
\begin{prop}
\label{p2.1}
There exists a unique linear map $\Lambda: \cz \cac^{1+}_3(V)
\to \cac_2^{1+}(V)$ such that
$$
\delta \Lambda  = \id_{\cz \cac_3^{1+}(V)}
\quad \mbox{ and } \quad \quad
\Lambda  \delta= \id_{\cac_2^{1+}(V)}.
$$
In other words, for any $h\in\cac^{1+}_3(V)$ such that $\der h=0$
there exists a unique $g=\laa(h)\in\cac_2^{1+}(V)$ such that $\der g=h$.
Furthermore, for any $\mu > 1$,
the map $\laa$ is continuous from $\cz \cac^{\mu}_3(V)$
to $\cac_2^{\mu}(V)$ and we have
\begin{equation*}%\label{ineqla}
\|\Lambda h\|_{\mu} \le \frac{1}{2^\mu-2} \|h\|_{\mu} ,\qquad h \in
\cz \cac^{\mu}_3(V).
\end{equation*}
\end{prop}

%\vspace{0.3cm}

It is worth mentioning at this point that
$\laa$ gives raise to a kind of generalized Young integral, which is a second
link between the algebraic structures introduced so far and a theory of generalized
integration:
\begin{cor}
\label{c2.1}
For any 1-increment $g\in\cac_2 (V)$ such that $\der g\in\cac_3^{1+}$,
%set
%$\delta f = (\id-\Lambda \delta) g$.
%Then
$$
(\id-\Lambda \delta) g = \lim_{|\Pi_{st}| \to 0} \sum_{i=0}^n g_{t_i\,t_{i+1}},
$$
where the limit is over any partition $\Pi_{st} = \{t_0=s,\dots,
t_n=t\}$ of $[s,t]$, whose mesh tends to zero. Thus by setting $\delta f=(\id-\Lambda \delta) g$, the
1-increment $\delta f$ is the indefinite integral of the 1-increment $g$.
\end{cor}
We can now explain heuristically how our generalized integral will be defined.
\begin{rem}
\label{r2.2}
Let $f$ and $g$ be two real valued smooth functions, and define $I\in\cac_2$
like in Remark \ref{r2.1}. Thanks to this remark and Proposition \ref{p2.1}, the following decomposition-recomposition
for $I=\int df\int dg$ holds true:
$$
\mbox{$\int dg\int df$ }\stackrel{\der}{\longrightarrow}
(\der g)\, (\der f)
\stackrel{\laa}{\longrightarrow}
\mbox{$\int dg\int df$ },
$$
where for the second step of this construction, we have only used the fact that the
product of increments $(\der g)\, (\der f)$, considered as an element of $\cz\cac_3$,
is smooth enough. This simple procedure allows then to extend the notion of iterated
integral to a non-smooth situation, by just applying the operator $\laa$ to
$(\der g)\, (\der f)$ whenever we are allowed to do it.
\end{rem}

\subsection{Some further notations}\label{sec:notations}
We summarize in this section some of the notation which will be
used throughout the paper.

\smallskip

A multilinear operator $A$ of order $l$, from
$\R^{d_1}\times\ldots\times\R^{d_l}$ to $\R^n$, is denoted as an
element $A\in\R^{n,d_1,\ldots,d_l}$. In order to avoid tricky
matrix notations, we have decided to expand all our computations
in coordinates, and use Einstein's convention on summations over
repeated indices. Notice that we will also use the notation
$A\in\R^{d_1,d_2,d_3,d_4}$ for a linear operator from
$\R^{d_3,d_4}$ to $\R^{d_1,d_2}$. We hope that this convention
won't lead to any ambiguity. The transposed of a matrix $M\in\R^{d_1,d_2}$ is written as $M^*$.

\smallskip

For a function $\varphi:(\R^n)^{q+1}\to\R$, we denote by $\partial_i^j\varphi(w_0,w_1,\ldots,w_q)$ the derivative of $\varphi$ with respect to the $i\textsuperscript{th}$ component of $w_j$, for $i\le n$ and $j=0,\ldots,q$.

\smallskip

We shall meet two kind of products of increments: first, for  $g\in\cac_n(\R^{l,d})$ and $h\in\cac_m(\R^d) $ we set $gh$ for the element of $\cac_{n+m-1}(\R^l)$ defined by
\begin{equation}\label{eq:convention-prod1}
(gh)_{t_1,\dots,t_{m+n-1}}= g_{t_1,\dots,t_{n}}
h_{t_{n},\dots,t_{m+n-1}}, \quad t_1,\dots,t_{m+n-1}\in\ott.
\end{equation}
If now $g\in\cac_n(\R^{l,d})$ and $h\in\cac_n(\R^d) $ we set $g\cdot h$ for the element of $\cac_{n}(\R^l)$ defined by
\begin{equation}\label{eq:convention-prod2}
(g \cdot h)_{t_1,\dots,t_{n}}=
g_{t_1,\dots,t_{n}} h_{t_{1},\dots,t_{n}},
\quad
t_1,\dots,t_{n}\in\ott.
\end{equation}

\smallskip

In order to avoid ambiguities, we shall denote by $\cn[f;\, \cac_j^\kappa]$ the $\kappa$-Hölder norm on the space $\cac_j$, for $j=1,2,3$. For $\zeta\in\cac_1(V)$, we also set $\mathcal{N}[\zeta;\mathcal{C}_{1}^\infty(V)]=\sup_{0\leq s \leq
T}|\zeta^i|_{V}$.

\smallskip

The integral of a real valued function $f$ with respect to another real valued function $g$, when properly defined, is written indistinctly as $\int f dg$ or $\cj(f  dg)$.

\section{The diffusion case}\label{sec:diff-case}

In this section, we will recall the basic steps which allow to define rigorously and solve an equation of the form:
\begin{equation}
\label{eq:sde-diffusion}
y_t=a+ \int_0^t \sigma(y_s) \, dx_s,\quad t\in[0,T],
\end{equation}
where $a$ is an initial condition in $\R^n$, $\si:\R^n\to\R^{n,d}$ is a smooth enough function, $T$ is an arbitrary positive constant, and $x$ is a generic $d$-dimensional noisy input with Hölder regularity $\ga>1/4$. In the algebraic integration setting \cite{Gu04,Gu06}, this task amounts to perform the following steps:
\begin{enumerate}
\item
Definition of an incremental operator $\delta$ and its inverse $\Lambda$.
\item
Definition of a suitable notion of controlled processes, and integration of those processes with respect to $x$.
\item
Resolution of the equation thanks to a fixed point procedure in the space of controlled processes.
\end{enumerate}
Having dealt with the first of those points at Section \ref{sec:basics-alg-integration}, we turn now to the second one, that is a definition of a useful notion of controlled processes.

\smallskip

\subsection{Weakly controlled processes}

Before giving the formal definition of a weakly controlled process in the context of equation (\ref{eq:sde-diffusion}), let us recall that when the regularity of the noise is $\ga>1/4$, the rough path setting relies on the a priori existence of an area (resp. volume) element $\mathbf{x^2}$ (resp. $\mathbf{x^3}$) satisfying the so-called Chen's relations:
\begin{hyp}
\label{h1.1} The path $\mathbb{R}^d$-valued $x$ is
$\gamma$-H\"older continuous with $\gamma>1/4$, and admits a L\'evy area and
a \emph{volume element}, that is two increments
$\mathbf{x}^\mathbf{2}\in
\mathcal{C}_{2}^{2\gamma}(\mathbb{R}^{d,d})$ and
$\mathbf{x}^\mathbf{3}\in
\mathcal{C}_2^{3\gamma}(\mathbb{R}^{d,d,d})$ (which represent respectively $\mathcal{J}(dxdx)$ and $\mathcal{J}(dxdxdx)$, with the conventions of Section \ref{sec:notations}) satisfying:
\begin{align*}
\delta \mathbf{x}^\mathbf{2}&=\delta x \otimes \delta x, \quad\text{
i.e. }(\delta (\mathbf{x^2})^{i j})_{sut}=(\delta x^i)_{su}(\delta
x^j)_{ut}\\
\delta \mathbf{x}^\mathbf{3}&=\mathbf{x^2}\otimes \delta x+\delta
x \otimes
\mathbf{x^2}, \quad\text{ i.e. }
(\delta
(\mathbf{x}^\mathbf{3})^{ijk})_{sut}=(\mathbf{x}^\mathbf{2}_{su})^{i
j}(\delta x^k)_{ut}+(\delta
x^i)_{su}(\mathbf{x}^\mathbf{2}_{ut})^{j k},
\end{align*}
for any $s,u,t\in [0,T]$, and any $i,j,k\in\{1,\ldots,d\}$.
\end{hyp}

The geometrical assumption for rough paths (which is satisfied by the fractional Brownian motion in the Stratonovich setting) also states that products of increments should be expressed in terms of iterated integrals:
\begin{hyp}
\label{h1.2} Let $\mathbf{x^2}$ be the area process defined at
Hypothesis \ref{h1.1}, and denote by $\mathbf{x^{2,s}}$ the
symmetric part of $\mathbf{x^2}$, i.e.
$\mathbf{x^{2,s}}=\frac{1}{2}(\mathbf{x^2}+(\mathbf{x^2})^\ast)$.
Then we suppose that for $0\leq s<t\leq T$ we have:
\begin{equation*}
\mathbf{x}^\mathbf{2,s}_{st}=\frac{1}{2}(\delta
x)_{st}\otimes(\delta x)_{st}.
\end{equation*}
\end{hyp}

\smallskip

With these hypotheses in mind, the natural class of processes which will be integrated against $x$ are processes whose increments can be expressed simply enough in terms of the increments of $x$:
\begin{defn}
\label{d1.1} Let $z$ be a process in
$\mathcal{C}_1^\kappa(\mathbb{R}^l)$ with $\kappa\leq \gamma$ and
$3\kappa+\gamma>1$, such that $z_0=a\in\mathbb{R}^l$. We say that $z$ is a weakly controlled path
based on $x$ if $\delta z\in
\mathcal{C}_{2}^\kappa(\mathbb{R}^l)$ can be decomposed into
\begin{equation}\label{1.1}
\delta z^i=(\zeta^1)^{i j}\delta x^j +
(\zeta^2)^{i j k} (\mathbf{x^2})^{k j}+r^i, \
\text{ i.e. } \ (\delta z^i)_{st}=(\zeta_s^1)^{i j}(\delta
x^j)_{st}+ (\zeta^2_{s})^{i j k}(\mathbf{x}^\mathbf{2}_{st})^{k
j}+r^i_{st},
\end{equation}
for any $1\leq i\leq l$, $1\leq j,k \leq d$. In the previous decomposition, we further assume that $\zeta^1\in
\mathcal{C}_1^\kappa(\mathbb{R}^{l,d})$ is a path with a given initial condition $\zeta^1_0=b\in
\mathbb{R}^{l,d}$, such that $\delta \zeta^1 \in
\mathcal{C}_2^\kappa(\mathbb{R}^{l,d})$ can be decomposed itself into:
\begin{equation*}
\delta (\zeta^1)^{i j}=(\zeta^2)^{i j k}\delta x^k+\rho^{i
j},\text{ i.e. }(\delta (\zeta^1)^{i j})_{st}=(\zeta_s^2)^{i j
k}(\delta x^k)_{st}+\rho_{st}^{i j},
\end{equation*}
for all $s,t\in[0,T]$, where $\zeta^2$ is a given path in $\mathcal{C}_{1}^\kappa(\mathbb{R}^{l,d,d})$. Notice also that in the previous equations, $r$ and $\rho$ are understood as regular remainders,  such that $r\in
\mathcal{C}_2^{3\kappa}(\mathbb{R}^l)$ and $\rho\in
\mathcal{C}_2^{2\kappa}(\mathbb{R}^{l,d})$.

\smallskip

The space of weakly controlled paths will be denoted by
$\mathcal{Q}_{\kappa,a,b}(\mathbb{R}^l)$, and a process $z\in
\mathcal{Q}_{\kappa,a,b}(\mathbb{R}^l)$ can be considered in fact
as a triple $(z,\zeta^1,\zeta^2)$. The natural semi-norm on
$\mathcal{Q}_{\kappa,a,b}(\mathbb{R}^l)$ is given by
\begin{align*}
\mathcal{N}[z;\mathcal{Q}_{\kappa,a,b}(\mathbb{R}^l)]&=\mathcal{N}[z;\mathcal{C}_1^\kappa(\mathbb{R}^l)]
+\mathcal{N}[\zeta^1;
\mathcal{C}_1^\infty(\mathbb{R}^{l,d})]+\mathcal{N}[\zeta^1;\mathcal{C}_1^\kappa(\mathbb{R}^{l,d})]\\
&+\mathcal{N}[\zeta^2;
\mathcal{C}_{1}^\infty(\mathbb{R}^{l,d,d})]+\mathcal{N}[\zeta^2;\mathcal{C}_{1}^\kappa(\mathbb{R}^{l,d,d})]\\
&+\mathcal{N}[\rho;\mathcal{C}_2^{2\kappa}(\mathbb{R}^{l,d})]+\mathcal{N}[r;\mathcal{C}_2^{3\kappa}(\mathbb{R}^l)],
\end{align*}
where the notations  $\mathcal{N}[g;\mathcal{C}_1^\kappa(V)]$ and $\mathcal{N}[\zeta;\mathcal{C}_{1}^\infty(V)]$ have been introduced at Section \ref{sec:notations}.
\end{defn}

\begin{rem}
With respect to the case $\ga>1/3$, the link between $\zeta^1$ and $\zeta^2$ in the definition of controlled processes is new. This {\it cascade} relation between $z$,  $\zeta^1$ and $\zeta^2$ is reminiscent of the Heinsenberg group structure of Lyons' theory, and is really natural for computational purposes.
\end{rem}
We can now study the stability of controlled processes by composition with a regular function.

\subsection{Composition of controlled processes}
The results of this section can be summarized into the following:
\begin{prop}
\label{p1.1} Assume Hypothesis \ref{h1.2} holds true. Let
$z\in\mathcal{Q}_{\kappa,a,b}(\mathbb{R}^l)$ with decomposition
(\ref{1.1}), consider a regular function $\varphi\in \mathcal{C}^3_b(\mathbb{R}^l;\mathbb{R})$
and set $\hat z=\varphi(z)$, $\hat a=\varphi(a)$, $\hat
b=\partial_i\varphi(a)b^i$. Then $\hat z\in
\mathcal{Q}_{\kappa,\hat a,\hat b}(\mathbb{R})$, and this latter path admits the decomposition
\begin{equation}
\label{decom} \delta \hat z=(\hat \zeta^1)^j\delta x^j+(\hat
\zeta^2)^{j k}(\mathbf{x^2})^{k j}+\hat r,
\end{equation}
with
$$
(\hat \zeta^1)^j=[\partial_i\varphi(z)\cdot(\zeta^1)^{i j}], \quad
(\hat \zeta^2)^{j k}=[\partial_i\varphi(z)\cdot(\zeta^2)^{i j k}]+[\partial_{i_1 i_2}\varphi(z)\cdot(\zeta^1)^{i_1 j}\cdot(\zeta^1)^{i_2 k}],
$$
and where $\hat r$ can be further decomposed into $\hat r=\hat r^1+\hat r^2 + \hat r^3$, with:
\begin{align*}
\hat r^1&=\partial_i\varphi(z) r^i, \\
\hat r^2&=
\frac{1}{2}[\partial_{i_1 i_2}\varphi(z)\cdot(\zeta^2)^{i_1 j_1 k_1}\cdot(\zeta^2)^{i_2 j_2 k_2}][(\mathbf{x^2})^{k_1 j_1}\cdot(\mathbf{x^2})^{k_2 j_2}]+\frac{1}{2}\partial_{i_1 i_2}\varphi(z)[r^{i_1}\cdot r^{i_2}]\\
&\quad+[\partial_{i_1 i_2}\varphi(z)\cdot(\zeta^1)^{i_1
j_1}\cdot(\zeta^2)^{i_2 j k}][\delta x^{j_1}
\cdot(\mathbf{x^2})^{k
j}]\\
&\quad+[\partial_{i_1 i_2}\varphi(z)\cdot(\zeta^1)^{i_1 j}][\delta
x^j\cdot r^{i_2}]+[\partial_{i_1 i_2}\varphi(z)\cdot(\zeta^2)^{i_1 j k}][(\mathbf{x^2})^{k j}\cdot r^{i_2}],\\
\hat r^3&=\delta\varphi(z)-\partial_i \varphi(z)\delta
z^i-\frac{1}{2}\partial_{i j}\varphi(z)[\delta z^i\cdot\delta
z^j].
\end{align*}
As far as $(\hat \zeta^1)^j$ is concerned, for $1\leq j\leq d$, it can be decomposed into
\begin{equation}
\label{auxdecom}
 \delta (\hat \zeta^1)^j=(\hat \zeta^2)^{j k}\delta x^k+\hat
\rho^j
\end{equation}
where the remainder $\hat \rho^j$ can be expressed as  $\hat \rho^j=(\hat \rho^1)^j+(\hat \rho^2)^j$, with:
\begin{align*}
(\hat \rho^1)^j&=\partial_i
\varphi(z)\rho^{i j}+[\delta[\partial_i\varphi(z)]\cdot\delta(\zeta^1)^{i j}]\\
&\quad+[\partial_{i_1 i_2}\varphi(z)\cdot(\zeta^1)^{i_1 j}\cdot(\zeta^2)^{i_2 j_2 k_2}](\mathbf{x}^\mathbf{2})^{k_2 j_2}+[\partial_{i_1 i_2}\varphi(z)\cdot(\zeta^1)^{i_1 j}]r^{i_2}, \\
(\hat \rho^2)^j&=(\zeta^1)^{i_1 j}\delta[\partial_{i_1}
\varphi(z)]-[(\zeta^1)^{i_1 j}\cdot\partial_{i_1
i_2}\varphi(z)]\delta z^{i_2}.
\end{align*}
Finally, the following cubical bound holds true for the norm of $\hat z$:
\begin{equation}
\label{1.1.1} \mathcal{N}[\hat z;\mathcal{Q}_{\kappa,\hat a,\hat
b}(\mathbb{R})]\leq
c_{\varphi,x,T}(1+\mathcal{N}^3[z;\mathcal{Q}_{\kappa,a,b}(\mathbb{R}^l)]).
\end{equation}
\end{prop}

\begin{proof}
This proof is a matter of long and tedious Taylor expansions, and we shall omit most of the details. Let us just mention that we start from the relation:
\begin{multline*}
(\delta \hat z)_{st}=\varphi(z_t)-\varphi(z_s)=\partial_i
\varphi(z_s)(\delta z^i)_{st}+\frac{1}{2}\partial_{i_1
i_2}\varphi(z_s)(\delta z^{i_1})_{st}(\delta
z^{i_2})_{st}\\
+\varphi(z_t)-\varphi(z_s)-\partial_i \varphi(z_s)(\delta
z^i)_{st}-\frac{1}{2}\partial_{i_1 i_2}\varphi(z_s)(\delta
z^{i_1})_{st}(\delta
z^{i_2})_{st}.
\end{multline*}
The desired decomposition (\ref{decom}) is then obtained by plugging relation (\ref{1.1}) into the last identity, and expanding further. It should also be noticed that some cancellations occur due to Hypothesis \ref{h1.2}. Relation (\ref{auxdecom}) is obtained in the same manner, and our bound (\ref{1.1.1}) is a matter of standard computations once the expressions (\ref{decom}) and (\ref{auxdecom}) are known.

\end{proof}

\subsection{Integration of controlled paths}
It is of course of fundamental importance for our purposes to be able to integrate a controlled process with respect to the driving signal $x$. This is achieved in the following proposition:
\begin{prop}
\label{p1.2} For a given $\gamma>1/4$ and $\kappa\leq\gamma$, let
$x$ be a process satisfying Hypothesis \ref{h1.1}. Let also
$m\in\mathcal{Q}_{\kappa,b,c}(\mathbb{R}^{1,d})$ with
decomposition $m_0=b\in \mathbb{R}^{1,d}$ and
\begin{equation}
\label{1.2} (\delta m^i)_{st}=(\mu_s^1)^{i j}(\delta
x^j)_{st}+(\mu^2_s)^{i j k}(\mathbf{x}^\mathbf{2}_{s t})^{k
j}+r_{st}^i, \quad 1\leq i \leq d,
\end{equation}
where $\mu^1\in
\mathcal{C}_1^\kappa(\mathbb{R}^{d,d})$, $\mu^1_0=c\in
\mathbb{R}^{d,d}$, and where $\delta\mu^1\in
\mathcal{C}_2^\kappa(\mathbb{R}^{d,d})$ can be decomposed into
\begin{equation}
(\delta (\mu^1)^{i j})_{st}=(\mu^2_s)^{i j k}(\delta
x^k)_{st}+\rho_{st}^{i j},
\end{equation}
with $\mu^2\in \mathcal{C}_1^\kappa(\mathbb{R}^{d,d,d})$,
$\rho\in \mathcal{C}_2^{2\kappa}(\mathbb{R}^{d,d})$,
$r\in\mathcal{C}_2^{3\kappa}(\mathbb{R}^{1,d})$.
Define then $z$ by $z_0=a\in \mathbb{R}$ and
\begin{equation}
\label{1.3} \delta z=m^i\delta x^i+(\mu^1)^{i
j}(\mathbf{x}^\mathbf{2})^{j i}+(\mu^2)^{i j
k}(\mathbf{x}^\mathbf{3})^{k j i}+\Lambda\big(r^i\delta
x^i+\rho^{i j}(\mathbf{x}^\mathbf{2})^{j i}+\delta (\mu^2)^{i j
k}(\mathbf{x}^\mathbf{3})^{k j i}\big).
\end{equation}
Finally, set
\begin{equation}
\label{1.4}\mathcal{J}(m^idx^i)=\delta z.
\end{equation}
Then,
\begin{enumerate}
\item[(i)]$z$ is well-defined as an element of
$\mathcal{Q}_{\kappa,a,b}(\mathbb{R})$, and $\mathcal{J}(m^idx^i)$ coincides with a Riemann integral in case of some smooth processes $m$ and $x$.

\item[(ii)]The semi-norm of $z$ in
$\mathcal{Q}_{\kappa,a,b}(\mathbb{R})$ can be estimated as
\begin{equation}
\label{1.5}
\mathcal{N}[z;\mathcal{Q}_{\kappa,a,b}(\mathbb{R})]\leq
c_{x,T}\{1+|b|_{\mathbb{R}^{1,d}}+T^{\gamma-\kappa}(|b|_{\mathbb{R}^{1,d}}+\mathcal{N}[m;\mathcal{Q}_{\kappa,b,c}(\mathbb{R}^{1,d})])\}.
\end{equation}
Furthermore, we obtain
\begin{align}
\label{1.5.1} \|\delta z\|_\kappa\leq
c_{T,x}T^{\gamma-\kappa}(|b|_{\mathbb{R}^{1,d}}+\mathcal{N}[m;\mathcal{Q}_{\kappa,b,c}(\mathbb{R}^{1,d})]).
\end{align}

\item[(iii)]It holds
\begin{align}
\label{1.6}
&\mathcal{J}_{st}(m^idx^i)\nonumber\\
&=\lim_{|\Pi_{st}\to0|}\sum_{q=0}^n[m^i_{t_{q}}(\delta
x^i)_{t_{q},t_{q+1}}+(\mu^1_{t_{q}})^{i
j}(\mathbf{x}^\mathbf{2}_{t_{q},t_{q+1}})^{j
i}+(\mu^2_{t_{q}})^{i j
k}(\mathbf{x}^\mathbf{3}_{t_{q},t_{q+1}})^{k j i}]
\end{align}
for any $0\leq s<t\leq T$, where the limit is taken over all
partitions $\Pi_{st}=\{t_0=s,\ldots,t_n=t\}$ of $[s,t]$, as the
mesh of the partition goes to zero.
\end{enumerate}
\end{prop}

\begin{proof}
Here again, the proof is long and cumbersome, and we prefer to avoid most of the technical details for sake of conciseness. Let us just try to justify the second part of the first assertion (about Riemann integrals).

\smallskip

Let us suppose then that $x$ is a smooth function and that $m\in
\mathcal{C}_1^\infty(\mathbb{R}^{1,d})$ admits the decomposition
(\ref{1.2}) with $\mu^1\in\mathcal{C}_1^\infty(\mathbb{R}^{d,d})$,
$\mu^2\in\mathcal{C}_1^\infty (\mathbb{R}^{d,d,d})$,
$\rho\in\mathcal{C}_2^\infty(\mathbb{R}^{d,d})$ and
$r\in\mathcal{C}_2^\infty(\mathbb{R}^{1,d})$. Then
$\mathcal{J}(m^i dx^i)$ is well-defined, and we have
$$
\int_s^tm_u^idx_u^i=m_s^i[x_t^i-x_s^i]+\int_s^t[m_u^i-m_s^i]dx_u^i
$$
for $s<t$, which can also be read as:
\begin{equation}
\label{1.7}\mathcal{J}(m^idx^i)=m^i\delta x^i + \mathcal{J}(\delta
m^i dx^i).
\end{equation}
Let us now plug the decomposition (\ref{1.2}) into the expression
(\ref{1.7}). This yields
\begin{align}
\label{1.8} \mathcal{J}(m^idx^i)&=m^i\delta
x^i+\mathcal{J}([(\mu^1)^{i j}\delta
x^j]dx^i)+\mathcal{J}([(\mu^2)^{i j k}(\mathbf{x}^\mathbf{2})^{k j}]dx^i)+\mathcal{J}(r^idx^i)\nonumber\\
&=m^i\delta x^i+(\mu^1)^{i j}(\mathbf{x}^\mathbf{2})^{j
i}+(\mu^2)^{i j k} (\mathbf{x}^\mathbf{3})^{k j
i}+\mathcal{J}(r^idx^i),
\end{align}
and observe that the terms $m^i\delta x^i$, $(\mu^1)^{i
j}(\mathbf{x}^\mathbf{2})^{j i}$ and $(\mu^2)^{i j k}
(\mathbf{x}^\mathbf{3})^{k j i}$ in (\ref{1.8}) are well-defined
provided that $x$, $\mathbf{x}^\mathbf{2}$ and
$\mathbf{x}^\mathbf{3}$ are defined themselves. To push forward
our analysis to the rough case, we still need to handle the term
$\mathcal{J}(r^idx^i)$. Owing to (\ref{1.8}) we can write
\begin{equation}
\label{1.8.1} \mathcal{J}(r^idx^i)=\mathcal{J}(m^i dx^i)-m^i\delta
x^i-(\mu^1)^{i j}(\mathbf{x}^\mathbf{2})^{j i}-(\mu^2)^{i j
k}(\mathbf{x}^\mathbf{3})^{k j i},
\end{equation}
and let us analyze this relation by applying $\delta$ to both sides of the last identity. Invoking standard rules on the operator $\delta$, and the fact that $x$ satisfies Hypothesis \ref{h1.1}, we end up with:
$$
\delta[\mathcal{J}(r^idx^i)]=\delta (\mu^1)^{i j}(\mathbf{x}^\mathbf{2})^{j i}+\delta(
\mu^2)^{i j k}(\mathbf{x}^\mathbf{3})^{k j i}-(\mu^2)^{i j
k}\delta x^k(\mathbf{x}^\mathbf{2})^{j i}+r^i\delta x^i,
$$
and thanks to the fact that $\delta (\mu^1)^{i j}=(\mu^2)^{i j k}\delta x^k+\rho^{i
j}$, we obtain:
\begin{equation}\label{1.9}
\delta[\mathcal{J}(r^idx^i)]=
\rho^{i j}(\mathbf{x}^\mathbf{2})^{j i}+ \delta (\mu^2)^{i j
k}(\mathbf{x}^\mathbf{3})^{k j i}+r^i\delta x^i.
\end{equation}
Assuming now that $\rho^{i j}(\mathbf{x}^\mathbf{2})^{j i},\;\delta
(\mu^2)^{i j k}(\mathbf{x}^\mathbf{3})^{k j i},\;r^i\delta x^i\in
\mathcal{C}_3^\nu$ with $\nu>1$, then $\rho^{i
j}(\mathbf{x}^\mathbf{2})^{j i}+ \delta (\mu^2)^{i j
k}(\mathbf{x}^\mathbf{3})^{k j i}+r^i\delta x^i$ becomes an
element of ${\rm Dom}(\Lambda)$. Thus, applying $\Lambda$ to both sides
of (\ref{1.9}) and inserting the result into (\ref{1.7}) we get
the expression (\ref{1.3}) of Proposition \ref{p1.2}. This justifies the fact that
(\ref{1.3}) is a natural expression for $\mathcal{J}(m^idx^i)$.

\end{proof}

As in \cite{NNRT06}, the previous proposition has a
straightforward multidimensional extension, which we state in the
following corollary:

\begin{cor}
\label{c1.1} Let $x$ be a process satisfying Hypothesis \ref{h1.1}
and let $m\in \mathcal{Q}_{\kappa,b,c}(\mathbb{R}^{l,d})$ with
decomposition $m_0=b \in \mathbb{R}^{l,d}$ and
\begin{equation}
\label{1.10} (\delta m^{i j})_{st}=(\mu_s^1)^{i j k}(\delta
x^k)_{st}+(\mu^2_s)^{i j k_1 k_2}(\mathbf{x}^\mathbf{2}_{s
t})^{k_2 k_1}+r_{st}^{i j}\, ;
\quad%\mbox{and}\quad
\delta (\mu_s^1)^{i j k_1}=(\mu^2_s)^{i j k_1 k_2}\delta
x^{k_2}+\rho^{i j k_1},
\end{equation}
where $(\mu^1)^{i j k_1}\in \mathcal{C}_1^\kappa(\mathbb{R})$, $(\mu^2)^{i j k_1 k_2}\in \mathcal{C}_1^\kappa(\mathbb{R})$,
$\rho^{i j k_1}\in \mathcal{C}_1^{2\kappa}(\mathbb{R})$ and $r^{i
j}\in \mathcal{C}_{2}^{3\kappa}(\mathbb{R})$, for $i=1,\ldots,l$ and
$j,k_1,k_2=1,\ldots,d$. Define $z$ by $z_0=a\in
\mathbb{R}^l$ and
\begin{align}
\label{1.11}\delta z^i &=\mathcal{J}(m^{i j} dx^j)\equiv m^{i
j}\delta
x^j+(\mu^1)^{i j k}(\mathbf{x}^\mathbf{2})^{k j}\nonumber\\
&+(\mu^2)^{i j k_1 k_2}(\mathbf{x}^\mathbf{3})^{k_2 k_1
j}+\Lambda\big(r^{i j}\delta x^j+\rho^{i j
k}(\mathbf{x}^\mathbf{2})^{k j}+\delta (\mu^2)^{i j k_1
k_2}(\mathbf{x}^\mathbf{3})^{k_2 k_1 j}\big).
\end{align}
Then the conclusions of Proposition \ref{p1.2} still hold in this
context.
\end{cor}

We also observe that our extended pathwise integral has a nice
continuity property with respect to the driving path $x$, whose
proof is also skipped here for sake of conciseness (see also
\cite[Proposition 4]{Gu04}, and \cite[Proposition 3.12]{NNRT06}).

\begin{prop}
\label{p1.3} Let $x$ be a function satisfying Hypotheses
\ref{h1.1} and \ref{h1.2}. Suppose that there exists a sequence
$\{x^n;\;n\geq1\}$ of piecewise $C^1$-functions from [0,T] to
$\mathbb{R}^d$ such that
\begin{equation*}
\label{1.12}
\lim_{n\to\infty}\mathcal{N}[x^n-x;\mathcal{C}_1^\gamma(\mathbb{R}^d)]=0,
\quad
\lim_{n\to\infty}\mathcal{N}[\mathbf{x}^{\mathbf{2},n}-\mathbf{x}^\mathbf{2};\mathcal{C}_2^{2\gamma}(\mathbb{R}^{d,d})]=0,
\end{equation*}
and $\lim_{n\to\infty}\mathcal{N}[\mathbf{x}^{\mathbf{3},n}-\mathbf{x}^\mathbf{3};\mathcal{C}_2^{3\gamma}(\mathbb{R}^{d,d,d})]=0$.
For $n\geq 1$, define $z^n\in \mathcal{C}_1^\kappa(\mathbb{R}^l)$
in the following way: set $z_0^n=b\in \mathbb{R}^l$ and assume that $\delta z^n$ can be decomposed into:
$$
\delta (z^n)^i=(\zeta^{1,n})^{i j}\delta x^j+(\zeta^{2,n})^{i j
k}(\mathbf{x}^\mathbf{2})^{k j}+(r^n)^i,
\quad
\delta (\zeta^{1,n})^{i j}=(\zeta^{2,n})^{i j k}\delta
x^k+(\rho^n)^{i j},
$$
for $1\leq i \leq l$ and $1\leq j,k\leq d$, where $\zeta^{1,n}\in
\mathcal{C}_1^\kappa(\mathbb{R}^{l,d})$ satisfies $\zeta^{1,n}_0=c
\in \mathbb{R}^{l,d}$, and
$\zeta^{2,n}\in \mathcal{C}_1^\kappa(\mathbb{R}^{l,d,d})$,
$\rho^n\in \mathcal{C}_2^{2\kappa}(\mathbb{R}^{l,d})$ and $r^n\in
\mathcal{C}_2^{3\kappa}(\mathbb{R}^l)$. Let also $z$ be a
weakly controlled process with decomposition (\ref{1.1}), such
that $z_0=b$, $\zeta^1_0=c$, and suppose that
\begin{align*}
\lim_{n\to\infty}&\big\{\mathcal{N}[z^n-z;\mathcal{C}_{1}^\kappa(\mathbb{R}^l)]+\mathcal{N}[\zeta^{1,n}-\zeta^1;\mathcal{C}_1^\infty(\mathbb{R}^{l,d})]
+\mathcal{N}[\zeta^{1,n}-\zeta^1;\mathcal{C}_1^\kappa(\mathbb{R}^{l,d})]\\
&+\mathcal{N}[\zeta^{2,n}-\zeta^2;\mathcal{C}_1^\infty(\mathbb{R}^{l,d,d})]
+\mathcal{N}[\zeta^{2,n}-\zeta^2;\mathcal{C}_1^\kappa(\mathbb{R}^{l,d,d})]\\
&+\mathcal{N}[\rho^n-\rho;\mathcal{C}_2^{2\kappa}(\mathbb{R}^{l,d})]+\mathcal{N}[r^n-r;\mathcal{C}_2^{3\kappa}(\mathbb{R}^l)]\big\}=0.
\end{align*}
Finally, let $\varphi:\mathbb{R}^l\to\mathbb{R}^{l',d}$ be a
$C_b^4$-function. Then
$$
\lim_{n\to\infty}\mathcal{N}[\mathcal{J}(\varphi(z^n)dx^n)-\mathcal{J}(\varphi(z)dx);\mathcal{C}_2^{\kappa}(\mathbb{R}^{l'})]=0.
$$
\end{prop}

\subsection{Rough diffusions equations}

In this section, we shall apply the previous considerations to
study differential equations driven by a rough signal, and recall that we first wish to solve simple equations of the form
\begin{equation}
\label{2.1} dy_t=\sigma(y_t)dx_t,\quad y_0=a,
\end{equation}
where $t\in[0,T]$, $y$ is a $\mathbb{R}^l$-valued continuous
process, $\sigma:\mathbb{R}^l\rightarrow \mathbb{R}^{l,d}$ is a
smooth enough function, $x$ is a $\mathbb{R}^d$-valued path and
$a\in\mathbb{R}^l$ is a fixed initial condition.

\smallskip

In our algebraic setting, we rephrase equation (\ref{2.1}) as follows: we
shall say that $y$ is a solution to (\ref{2.1}), if $y_0=a$, $y\in
\mathcal{Q}_{\kappa,a,\sigma(a)}(\mathbb{R}^l)$ and for any $0\leq
s\leq t\leq T$ we have
\begin{equation}
\label{2.2} (\delta y)_{st}=\mathcal{J}_{st}(\sigma(y)dx),
\end{equation}
where the integral $\mathcal{J}(\sigma(y)dx)$ has to be understood
in the sense of Corollary \ref{c1.1}.

\smallskip

With these notations in mind, our existence and uniqueness result is the following:
\begin{thm}
\label{t2.1}Let $x$ be a process satisfying Hypotheses \ref{h1.1}
and \ref{h1.2}, and $\sigma:\mathbb{R}^l\rightarrow
\mathbb{R}^{l,d}$ be a $C^4_b$-function. Then
\begin{itemize}
\item[(i)]Equation (\ref{2.2}) admits a unique solution $y$ in
$\mathcal{Q}_{\kappa,a,\sigma(a)}(\mathbb{R}^l)$ for any
$\kappa<\gamma$ such that $3\kappa+\gamma>1$.

\item[(ii)]The mapping
$(a,x,\mathbf{x}^\mathbf{2},\mathbf{x}^\mathbf{3})\mapsto y$ is
continuous from
$$
\mathbb{R}^l\times\mathcal{C}_1^\gamma(\mathbb{R}^d)\times
\mathcal{C}_2^{2\gamma}(\mathbb{R}^{d,d})\times
\mathcal{C}_2^{3\gamma}(\mathbb{R}^{d,d,d})\text{ to }
\mathcal{C}_1^\kappa(\mathbb{R}^l),
$$
in the following sense: let $z$ be the unique solution of (\ref{2.2}) in
$\mathcal{Q}_{\kappa,a,\sigma(a)}(\mathbb{R}^l)$ and $\tilde z$ the unique solution of (\ref{2.2}) in $\mathcal{Q}_{\kappa,\tilde a,\sigma(\tilde a)}(\mathbb{R}^l)$,
based on $x,\tilde x$, respectively. Then, there exists a positive
constant $\hat c_{\sigma,x,\tilde x}$ depending only on
$\sigma,x,\tilde x$ such that
\begin{multline*}
\mathcal{N}[z-\tilde z;\mathcal{C}_1^\kappa(\mathbb{R}^l)]
\leq \hat c_{x,\tilde x}\big\{|a-\tilde a|+\mathcal{N}[x-\tilde
x;\mathcal{C}_1^\gamma(\mathbb{R}^l)]\\
+\mathcal{N}[\mathbf{x}^\mathbf{2}-\mathbf{\tilde
x}^\mathbf{2};\mathcal{C}_2^{2\gamma}(\mathbb{R}^{d,d})]
+\mathcal{N}[\mathbf{x}^\mathbf{3}-\mathbf{\tilde
x}^\mathbf{3};\mathcal{C}_2^{3\gamma}(\mathbb{R}^{d,d,d})]\big\}.
\end{multline*}
\end{itemize}
\end{thm}

\begin{proof}
As in \cite{Gu04,GT08}, we first identify the solution on a small interval
$[0,\tau]$ as the fixed point of the map
$\Gamma:\mathcal{Q}_{\kappa,a,\sigma(a)}(\mathbb{R}^l)\rightarrow
\mathcal{Q}_{\kappa,a,\sigma(a)}(\mathbb{R}^l)$ defined by
$\Gamma(z)=\hat z$ with $\hat z_0=a$ and $\delta \hat
z=\mathcal{J}(\sigma(z)dx)$. The first step in this direction is
to show that the ball
\begin{equation}
\label{2.3}
B_M=\{z;\;z_0=a,\;\mathcal{N}[z;\mathcal{Q}_{\kappa,a,\sigma(a)}([0,\tau];\,\mathbb{R}^l)]\leq
M\}
\end{equation}
is invariant under $\Gamma$ if $\tau$ is small enough and $M$ is
large enough. However, due to Corollary \ref{c1.1} and  Proposition
\ref{p1.1}, invoking the fact that $\sigma$ is bounded together with its derivaties
and assuming $\tau\leq 1$, we obtain
\begin{align}
\label{2.4}
\mathcal{N}[\Gamma(z);\mathcal{Q}_{\kappa,a,\sigma(a)}(\mathbb{R}^l)]&\leq
c_{x}\{1+|\sigma(a)|_{\mathbb{R}^{l,d}}+\tau^{\gamma-\kappa}(|\sigma(a)|_{\mathbb{R}^{l,d}}+\mathcal{N}[\sigma(z);\mathcal{Q}_{\kappa,\hat a,\hat b}(\mathbb{R}^{l,d})])\}\nonumber\\
&\leq
c_{x,\sigma}\{1+\tau^{\gamma-\kappa}\mathcal{N}[\sigma(z);\mathcal{Q}_{\kappa,\hat a, \hat b}(\mathbb{R}^{l,d})]\}\nonumber\\
&\leq c_{x,\sigma}\{1+\tau^{\gamma-\kappa}(1+\mathcal{N}^3[z;\mathcal{Q}_{\kappa,a,\sigma(a)}(\mathbb{R}^l)])\}\nonumber\\
&\leq
\tilde{c}_{x,\sigma}\{1+\tau^{\gamma-\kappa}\mathcal{N}^3[z;\mathcal{Q}_{\kappa,a,\sigma(a)}(\mathbb{R}^l)]\},
\end{align}
where $\hat a=\sigma(a)$ and  $\hat b=\partial_i\sigma(a)\sigma^{i
\cdot}(a)$.
Taking $M>\tilde{c}_{x,\sigma}$ and
$\tau\leq\tau_0=\big(\frac{1}{M^2\tilde{c}_{x,\sigma}}-\frac{1}{M^3}\big)^\frac{1}{\gamma-\kappa}\wedge
1$, we obtain that
$\tilde{c}_{x,\sigma}(1+\tau^{\gamma-\kappa}M^3)\leq M$.
Therefore, the ball $B_M$ defined at (\ref{2.3}) is left invariant
by $\Gamma$.

\smallskip

It is now a matter of standard considerations to settle a fixed
point argument for $\Gamma$ on $[0,\tau]$, and also to patch
solutions on any interval of the form $[k\tau,(k+1)\tau]$ for
$k\ge 1$. The details of this procedure are left to the reader.

\end{proof}

\section{The delay equation case}\label{sec:delay}

This section is devoted to show how to change the diffusion setting in order to cover the case of delayed systems, having in mind to solve an equation of the form:
\begin{equation}
\label{def-delay-eq} \left\{
\begin{array}{ll}
dy_t=\sigma(y_t,y_{t-r_1},\ldots,y_{t-r_{q}}) \, dx_t\quad  t\in[0,T],    \\
y_t=\xi_t,  \qquad  t \in [-r_{q},0],
\end{array}
\right.
\end{equation}
where $x$ is $\mathbb{R}^d$-valued $\gamma$-H\"older continuous
function with $\gamma>1/4$, the function $\sigma$\ is smooth
enough, $\xi$ is a $\mathbb{R}^n$-valued $3\gamma$-H\"older
continuous function, and $0<r_1<\ldots<r_{q}<\infty$. Notice that
for notational convenience, we set $r_0=0$ and we shall also use
the notation
\begin{equation}
\label{def-s(y)} \mathfrak{s}(y)_t=( y_{t-r_1}, \ldots,
y_{t-r_{q}}), \qquad t \in [0,T],
\end{equation}
which means that equation (\ref{def-delay-eq}) can be written as:
$$
\left\{
\begin{array}{ll}
dy_t=\sigma(y_{t-r_0},\mathfrak{s}(y)_t) \, dx_t\quad  t\in[0,T],    \\
y_t=\xi_t,  \qquad  t \in [-r_{q},0].
\end{array}
\right.
$$
As in \cite{NNT07}, the main ingredient in order to go from the diffusion to the delayed case will be the introduction of a new class of processes, namely the class of delayed controlled paths, which captures the structure of our equation. We shall thus first define this new class of paths, and see how to integrate them with respect to the driving process $x$.

\subsection{Delayed controlled paths}
As in the diffusion case, our analysis will rely on some a priori
increments based on our driving noise $x$. More specifically, we
set $\delta (x(v))_{st}\triangleq (\der x)_{s-v,t-v}$ for
$s,t,v\in[0,T]$, and we assume the following:

\begin{hyp}
\label{del:h1.1} The path $x$ is a $\mathbb{R}^d$-valued
$\gamma$-H\"older continuous function with $\gamma>1/4$, and
admits two doubly delayed L\'evy areas and two doubly delayed
volume elements. Namely, for $v,v'\in\{r_{q},\ldots,r_0\}$, we
assume that there exist four paths
$$
\mathbf{x}^\mathbf{2}(v',v),\mathbf{x}^\mathbf{2}(v'-v,v)\in
\mathcal{C}_{2}^{2\gamma}([0,T];\mathbb{R}^{d,d}),
\quad
\mathbf{x}^\mathbf{3}(v',v),\mathbf{x}^{\mathbf{3}}(v'-v,v)\in
\mathcal{C}_2^{3\gamma}([0,T];\mathbb{R}^{d,d,d}),
$$
satisfying the relations $\delta \mathbf{x}^\mathbf{2}(v'',v)=\delta (x(v+v'')) \otimes
\delta (x(v))$ and
$$
\delta \mathbf{x}^\mathbf{3}(v'',v)=\mathbf{x^2}(v'',v)\otimes
\delta x+\delta (x(v+v'')) \otimes
\mathbf{x^2}(v,r_0),
$$
which can also be written as:
\begin{eqnarray*}
(\delta (\mathbf{x^2}(v'',v))^{i j})_{sut}&=&
(\delta x^i)_{s-v-v'',u-v-v''}(\delta x^j)_{u-v,t-v}  \\
(\delta (\mathbf{x}^\mathbf{3}(v'',v))^{i j k})_{sut}&=&(\mathbf{x}^\mathbf{2}_{su}(v'',v))^{i j}(\delta x^k)_{ut}+(\delta x^i)_{s-v-v'',u-v-v''}(\mathbf{x}^\mathbf{2}_{u
t}(v,r_0))^{j k},
\end{eqnarray*}
for $v''=v' \text{ or } v'-v$, for any $s,u,t\in [0,T]$, and any
$i,j,k\in\{1,\ldots,d\}$. The following notational simplification
will also be used in the sequel: we may set
$\mathbf{x}^\mathbf{2}(v''):=\mathbf{x}^\mathbf{2}(v'',v)$
whenever $v=r_0$.
\end{hyp}

\begin{rem}
This hypothesis takes a more complex form than in \cite{NNT07}, where the case $\ga>1/3$ was treated. However, in case of a regular process $x$, it should be noticed that the increments $\mathbf{x^2}(v_1,v_2)$ and $\mathbf{x^3}(v_1,v_2)$ can be defined as:
$$
\mathbf{x}_{st}^\mathbf{2}(v_1,v_2)=\int_{s-v_2}^{t-v_2} (\delta
x(v_1))_{s-v_2,w}  \otimes dx_w , \quad\mbox{and}\quad
\mathbf{x}_{st}^\mathbf{3}(v_1,v_2)=\int_{s}^{t}
\mathbf{x}_{sw}^\mathbf{2}(v_1,v_2) \otimes dx_w,
$$
which means that $\mathbf{x^2}$ (resp. $\mathbf{x^3}$) takes the usual form of a double (resp. triple) iterated integral.
\end{rem}

As in Hypothesis \ref{h1.2}, one should also express the fact that products of increments can be expressed in terms of iterated integrals. The following hypothesis is then easily shown to be a natural extension of what can be obtained in case of a smooth function $x$:
\begin{hyp}
\label{del:h1.2} For $v,v'\in\{r_q,\ldots,r_0\}$, let
$\mathbf{x^2}(v',v)$ and $\mathbf{x^2}(v'-v,v)$
 be the area processes defined at
Hypothesis \ref{del:h1.1}. Then we suppose that for all $0\leq
s<t\leq T$, we have
$\mathbf{x}_{st}^\mathbf{2}(v',v)=\mathbf{x}_{s-v,t-v}^\mathbf{2}(v',r_0)$
and
\begin{equation*}
[\delta x(v)]_{st}\otimes[\delta
x(v')]_{st}=\mathbf{x}^\mathbf{2}_{st}(v-v',v')+(\mathbf{x}^\mathbf{2}_{st}(v'-v,v))^{\ast}.
\end{equation*}
\end{hyp}

\smallskip

With these hypotheses in hand, the delay equation will be solved
in the space of \emph{delayed controlled processes}, which can be
defined as follows:
\begin{defn}
\label{del:d1.1} Let $-\infty<a <b\leq T$, a given initial datum  $\alpha\in\mathbb{R}^n$ and
$z\in\mathcal{C}_1^\kappa([a,b];\mathbb{R}^n)$ with $\kappa\leq
\gamma$ and $3\kappa+\gamma>1$. We say that $z$ is a delayed
controlled path based on $x$ if $z_a=\alpha$, and if $\delta z\in
\mathcal{C}_{2}^\kappa([a,b];\mathbb{R}^n)$ can be decomposed into
\begin{equation}
\label{del:1.1} \delta z^i=(\zeta^1)^{i j}\delta x^j +
(\zeta^{(2,i')})^{i j k} (\mathbf{x^2}(r_{i'}))^{k j}+\mathcal{R}^i,
\end{equation}
for all $1\leq i\leq n$, and where the index $i'$ is summed over the set $0\leq i'\leq q$. Just as in Definition \ref{d1.1}, the process $\zeta^1$ above has to admit the further decomposition: $\zeta^1_a=\beta\in \mathbb{R}^{n,d}$, where $\beta$ has to be interpreted as another initial datum, and for $1\leq j,k \leq d$
\begin{equation}\label{del:1.1b}
\delta (\zeta^1)^{i j}=(\zeta^{(2,i')})^{i j k}\delta
(x(r_{i'}))^k+\rho^{i j}.
\end{equation}
The regularity of the processes introduced above has to be the following:
$\zeta^1$ is an element of $\mathcal{C}_1^\kappa ([a,b];\mathbb{R}^{n,d})$, $\zeta^{(2,i')}\in
\mathcal{C}_{1}^\kappa([a,b];\mathbb{R}^{n,d,d})$, and the remainders
$\mathcal{R}$, $\rho$ must satisfy $\mathcal{R}\in
\mathcal{C}_2^{3\kappa}([a,b];\mathbb{R}^n)$ and $\rho\in
\mathcal{C}_2^{2\kappa}([a,b];\mathbb{R}^{n,d})$.

\smallskip

The space of delayed controlled paths on $[a,b]$ will be denoted
by $\mathcal{D}_{\kappa,\alpha,\beta}([a,b];\mathbb{R}^n)$, and a
path $z\in \mathcal{D}_{\kappa,\alpha,\beta}([a,b];\mathbb{R}^n)$
should be considered in fact as a $(q+3)$-tuple
$
(z,\zeta^1,\zeta^{(2,0)},\ldots,$ $\zeta^{(2,q)}).
$
The natural semi-norm on
$\mathcal{D}_{\kappa,\alpha,\beta}([a,b];\mathbb{R}^n)$ is then given
by
\begin{align*}
&\mathcal{N}[z;\mathcal{D}_{\kappa,\alpha,\beta}([a,b];\mathbb{R}^n)]\\
&=\mathcal{N}[z;\mathcal{C}_1^\kappa([a,b];\mathbb{R}^n)]
+\mathcal{N}[\zeta^1;
\mathcal{C}_1^\infty([a,b];\mathbb{R}^{n,d})]+\mathcal{N}[\zeta^1;\mathcal{C}_1^\kappa([a,b];\mathbb{R}^{n,d})]\\
&+\sum_{i'=0}^{q}\mathcal{N}[\zeta^{(2,i')};
\mathcal{C}_{1}^\infty([a,b];\mathbb{R}^{n,d,d})]+\sum_{i'=0}^{q}\mathcal{N}[\zeta^{(2,i')};\mathcal{C}_{1}^\kappa([a,b];\mathbb{R}^{n,d,d})]\\
&+\mathcal{N}[\rho;\mathcal{C}_2^{2\kappa}([a,b];\mathbb{R}^{n,d})]+\mathcal{N}[\mathcal{R};\mathcal{C}_2^{3\kappa}([a,b];\mathbb{R}^n)],
\end{align*}
where we recall that the notations $\mathcal{N}[g;\mathcal{C}_1^\kappa([a,b];V)]$ and $\mathcal{N}[g;\mathcal{C}_{1}^\infty([a,b];V)]$ have been introduced at Section \ref{sec:notations}.
\end{defn}

\smallskip

Unfortunately, the structure above is not sufficient in order to solve the fractional delay equation for $\ga\le 1/3$, and an additional notion of \emph{doubly delayed controlled processes} has to be introduced.
\begin{defn}
\label{del:d1.2} Let $-\infty<a <b\leq T$, a given initial datum  $\hat\alpha\in\mathbb{R}^n$ and
$z\in\mathcal{C}_1^\kappa([a,b];\mathbb{R}^n)$ with $\kappa\leq
\gamma$ and $3\kappa+\gamma>1$. We say that $z$ is a doubly
delayed controlled path based on $x$, if $z_a=\hat \alpha$ and if $\delta z\in
\mathcal{C}_{2}^\kappa([a,b];\mathbb{R}^n)$ can be decomposed into
\begin{multline}
\label{del:1.1.2} \delta z^i=(\zeta^{(1,i'')})^{i j}\delta
(x(r_{i''}))^j + (\zeta^{(2,i',j')})^{i j k}
(\mathbf{x^2}(r_{j'},r_{i'}))^{k
j} \\
+(\zeta^{(3,i'',j'')})^{i j k} (\mathbf{x^2}(r_{j''}-r_{i''},r_{i''}))^{k j}+\mathcal{R}^i,
\end{multline}
for all $1\leq i\leq n$, $1\leq j,k \leq d$, and where the indices
$i',j'$ and $i'',j''$ are summed over the set $\{1,\ldots,q\}$ and
$\{0,1,\ldots,q\}$, respectively. As in Definition \ref{del:d1.1},
the processes $\zeta^{(1,i'')}$ above have to admit the further
decomposition: $\zeta^{(1,i'')}_a=\hat\beta^{(i'')}$, where
$\hat\beta^{(i'')}\in \mathbb{R}^{n,d}$ has to be interpreted as
another initial datum, and
\begin{equation}
\begin{array}{ll}
\label{del:1.1.3} \delta (\zeta^{(1,0)})^{i
j}&=(\zeta^{(3,0,j'')})^{i j k}\delta
(x(r_{j''}))^k+(\rho^{(0)})^{i j},\\
\delta (\zeta^{(1,i')})^{i j}&=(\zeta^{(2,i',j')})^{i j k}\delta
(x(r_{i'}+r_{j'}))^k+(\zeta^{(3,i',j'')})^{i j k}\delta
(x(r_{j''}))^k+(\rho^{(i')})^{i j}.
\end{array}
\end{equation}
The regularity we ask for the processes introduced above is the
following: $\zeta^{(1,i'')}$ is an element of
$\mathcal{C}_1^\kappa ([a,b];\mathbb{R}^{n,d})$, we have
$\zeta^{(2,i',j')},\zeta^{(3,i'',j'')}\in
\mathcal{C}_{1}^\kappa([a,b];\mathbb{R}^{n,d,d})$, and the
remainders $\mathcal{R}$, $\rho^{(i'')}$ satisfy $\mathcal{R}\in
\mathcal{C}_2^{3\kappa}([a,b];\mathbb{R}^n)$ and $\rho^{(i'')}\in
\mathcal{C}_2^{2\kappa}([a,b];\mathbb{R}^{n,d})$.

\smallskip

The space of doubly delayed controlled paths on the interval
$[a,b]$ will be denoted by $\widehat{\mathcal{D}}_{\kappa,\hat
\alpha,\hat \beta}([a,b];\mathbb{R}^n)$, and a path $z\in
\widehat{\mathcal{D}}_{\kappa,\hat \alpha,\hat
\beta}([a,b];\mathbb{R}^n)$ should be considered in fact as a $(2
q^2+3q+3)$-tuple
$
(z,\{\zeta^{(1,i'')}\},\{\zeta^{(2,i',j')}\},\{\zeta^{(3,i'',j'')}\}).
$
The natural semi-norm on $\widehat{\mathcal{D}}_{\kappa,\hat
\alpha,\hat \beta}([a,b];\mathbb{R}^n)$ is given by
\begin{align*}
&\mathcal{N}[z;\widehat{\mathcal{D}}_{\kappa,\hat \alpha,\hat \beta}([a,b];\mathbb{R}^n)]\\
&=\mathcal{N}[z;\mathcal{C}_1^\kappa([a,b];\mathbb{R}^n)]
+\sum_{i'=0}^{q}\mathcal{N}[\zeta^{(1,i')};
\mathcal{C}_1^\infty([a,b];\mathbb{R}^{n,d})]+\sum_{i'=0}^{q}\mathcal{N}[\zeta^{(1,i')};\mathcal{C}_1^\kappa([a,b];\mathbb{R}^{n,d})]\\
&+\sum_{i',j'=1}^{q}\mathcal{N}[\zeta^{(2,i',j')};
\mathcal{C}_{1}^\infty([a,b];\mathbb{R}^{n,d,d})]+\sum_{i',j'=1}^{q}\mathcal{N}[\zeta^{(2,i',j')};\mathcal{C}_{1}^\kappa([a,b];\mathbb{R}^{n,d,d})]\\
&+\sum_{i',j'=0}^{q}\mathcal{N}[\zeta^{(3,i',j')};
\mathcal{C}_{1}^\infty([a,b];\mathbb{R}^{n,d,d})]+\sum_{i',j'=0}^{q}\mathcal{N}[\zeta^{(3,i',j')};\mathcal{C}_{1}^\kappa([a,b];\mathbb{R}^{n,d,d})]\\
&+\mathcal{N}[\rho;\mathcal{C}_2^{2\kappa}([a,b];\mathbb{R}^{n,d})]+\mathcal{N}[\mathcal{R};\mathcal{C}_2^{3\kappa}([a,b];\mathbb{R}^n)].
\end{align*}
\end{defn}

As in Section \ref{sec:diff-case}, we shall now see how delayed controlled paths behave under composition with a smooth map, and also how to integrate them with respect to the driving process $x$.

\subsection{Composition of delayed controlled processes}
The composition of a delayed process with a smooth map like the function $\si$ appearing in equation (\ref{def-delay-eq}) gives raise to a doubly delayed controlled process. This fact is detailed in the following proposition, for which we recall a notation: for a function $\varphi:(\R^n)^{q+1}\to\R$, we denote by $\partial_i^j\varphi(w_0,w_1,\ldots,w_q)$ the derivative of $\varphi$ with respect to the $i\textsuperscript{th}$ component of $w_j$, for $i\le n$ and $j=0,\ldots,q$.
\begin{prop}
\label{del:p1.1} Assume Hypothesis \ref{del:h1.2} holds true. Consider $0\leq a <b\leq T$,
let $\alpha,\tilde \alpha$ and $\beta,\tilde \beta$ be two initial
conditions, respectively in $\mathbb{R}^n$ and $\mathbb{R}^{n,d}$,
and let also $\varphi$ be a function in $C_b^3(\mathbb{R}^{n,q+1}; \mathbb{R})$. Define a map $T_\varphi$ on
$\mathcal{D}_{\kappa,\alpha,\beta}([a,b];\mathbb{R}^n)\times
\mathcal{D}_{\kappa,\tilde \alpha,\tilde
\beta}([a-r_{q},b-r_1];\mathbb{R}^n)$ by $T_\varphi(z,\tilde
z)=\hat z$, with
$\hat z_t=\varphi(z_t,\mathfrak{s}(\tilde z)_t)$,
for all $t\in[a,b]$. Then, setting
$$
\hat \alpha=\varphi(\alpha,\mathfrak{s}(\tilde
z)_a),
\quad
\hat
\beta^{(0)}=\partial_i^0\varphi(\alpha,\mathfrak{s}(\tilde
z)_a)\beta^i,
\quad
\hat
\beta^{(i')}=\partial_i^{i'}\varphi(\alpha,\mathfrak{s}(\tilde
z)_a) (\tilde\zeta^1_{a-r_{i'}})^i,
$$
for all $1\leq i'\leq q$, we have that $T_\varphi(z,\tilde z)\in
\widehat{\mathcal{D}}_{\kappa,\hat \alpha,\hat
\beta}([a,b];\mathbb{R})$. Moreover, the following cubic bound holds true:
\begin{align}
\label{del:1.1.1} &\mathcal{N}[\hat
z;\widehat{\mathcal{D}}_{\kappa,\hat \alpha,\hat
\beta}([a,b];\mathbb{R})]\nonumber\\
&\leq
c_{\varphi,x,T}(1+\mathcal{N}^3[z;\mathcal{D}_{\kappa,\alpha,\beta}([a,b];\mathbb{R}^n)]+\mathcal{N}^3[\tilde
z;\mathcal{D}_{\kappa,\tilde \alpha,\tilde
\beta}([a-r_{q},b-r_1];\mathbb{R}^n)]).
\end{align}
\end{prop}

\begin{proof}
As in Proposition \ref{p1.1}, the proof of this result is based on some cumbersome Taylor expansions which won't be detailed here. Let us just mention briefly the decomposition we obtain for $\hat z$: observe that  $\delta\hat z$ can be decomposed into
\begin{align}
\label{del:decom} &(\delta \hat z)_{st}=(\hat
\zeta^{(1,i')}_s)^{j}(\delta x^j)_{s-r_{i'},t-r_{i'}}+ (\hat
\zeta^{(2,i',j')}_{s})^{j
k}(\mathbf{x}^\mathbf{2}_{st}(r_{j'},r_{i'}))^{k
j}\Ind_{\{i',j'\neq 0\}}\nonumber\\
&\qquad+(\hat \zeta^{(3,i',j')}_{s})^{j
k}(\mathbf{x}^\mathbf{2}_{st}(r_{j'}-r_{i'},r_{i'}))^{k j}+
\mathcal{\hat R}_{st},
\end{align}
where we recall that Einstein's convention on repeated indices is
used, and where $j,k$ are summed over the set $\{1,\ldots,d\}$,
and $i',j'$ over the set $\{0,1,\ldots,q\}$. Furthermore, the
expression for the coefficients $\hat\zeta$ is given by:
\begin{align*}
(\hat
\zeta^{(1,i')}_s)^j&=[\partial_i^{i'}\varphi(z,\mathfrak{s}(\tilde
z))\cdot(\tilde \zeta^1(r_{i'}))^{i j}]_s,\\
(\hat \zeta^{(2,i',j')}_s)^{j
k}&=[\partial_i^{i'}\varphi(z,\mathfrak{s}(\tilde z))\cdot(\tilde
\zeta^{(2,j')}(r_{i'}))^{i j k}]_s,
\end{align*}
and
\begin{multline*}
(\hat \zeta^{(3,i',j')}_s)^{j k}=[\partial_{i_1 i_2}^{i' j'}
\varphi(z,\mathfrak{s}(\tilde z))\cdot(\tilde
\zeta^1(r_{i'}))^{i_1 j}\cdot(\tilde \zeta^1(r_{j'}))^{i_2 k}]_s \\
+[\partial_i^{i'}\varphi(z,\mathfrak{s}(\tilde z))(\tilde
\zeta^{(2,0)}(r_{i'}))^{i j k}]_s \Ind_{\{i' =
j'\}}+[\partial_i^{0}\varphi(z,\mathfrak{s}(\tilde
z))(\zeta^{(2,j')})^{i j k}]_s \Ind_{\{i' = 0,j'\neq 0\}},
\end{multline*}
with the convention that $\tilde \zeta^1(r_0)\triangleq \zeta^1$
and $\tilde \zeta^{(2,j')}(r_0)\triangleq \zeta^{(2,j')}$,
respectively.

As far as the decomposition of $(\hat \zeta^{(1,i')})^j$ is
concerned, we obtain, for $1\leq j\leq d$:
\begin{multline}
\label{del:auxdecom}
(\delta
(\hat\zeta^{(1,i')})^{j})_{st}=\Ind_{\{i',j'\ne
0\}}(\hat\zeta^{(2,i',j')}_s)^{j k}(\delta
x^k)_{s-r_{i'}-r_{j'},t-r_{i'}-r_{j'}}\\
+(\hat\zeta^{(3,i',j')}_s)^{j k}(\delta
x^k)_{s-r_{j'},t-r_{j'}}+(\hat\rho^{(i')}_{st})^{j}.
\end{multline}
For sake of conciseness, we don't include the (long) expressions
we have obtained for the remainders $\mathcal{\hat R}_{st}$,
$(\hat \rho^{(i')}_{st})^{j}$ here. They are also obtained via a
Taylor type expansion, and the relation $\mathbf{x}_{st}^\mathbf{2}(v',v)=\mathbf{x}_{s-v,t-v}^\mathbf{2}(v',r_0)$ (imposed by Hypothesis \ref{del:h1.2}) turns out to be useful at some points of the computations.

\end{proof}

\smallskip

It should also be mentioned that, for a fixed $\tilde z$, the map $T_\varphi(\cdot,\tilde
z):\mathcal{D}_{\kappa,\alpha,\beta}([a,b];\mathbb{R}^n)\rightarrow
\widehat{\mathcal{D}}_{\kappa,\hat \alpha,\hat
\beta}([a,b];\mathbb{R})$ is locally Lipschitz continuous:

\begin{prop}
\label{del:p1.2} Let the notation of Proposition \ref{del:p1.1} prevail, and
suppose that $\varphi$ is a function in $C_b^4(\mathbb{R}^{n,q+1};\mathbb{R})$. Let $0\leq a<b\leq T$, let
$z^{(1)},z^{(2)}\in
\mathcal{D}_{\kappa,\alpha,\beta}([a,b];\mathbb{R}^n)$ and let
$\tilde z \in \mathcal{D}_{\kappa,\tilde \alpha,\tilde
\beta}([a-r_k,b-r_1];\mathbb{R}^n)$. Then,
\begin{align}
\label{del:lips} &\mathcal{N}[T_\varphi(z^{(1)},\tilde
z)-T_\varphi(z^{(2)},\tilde
z);\widehat{\mathcal{D}}_{\kappa,0,0}([a,b];\mathbb{R})]\notag\\
&\leq c_{x,\varphi,T}(1+C(z^{(1)},z^{(2)},\tilde
z))^3\mathcal{N}[z^{(1)}-z^{(2)};\mathcal{D}_{\kappa,0,0}([a,b];\mathbb{R}^n)],
\end{align}
where
\begin{align}
\label{del:cons-lips} C(z^{(1)},z^{(2)},\tilde
z)&=\mathcal{N}[\tilde z;\mathcal{D}_{\kappa,\tilde \alpha,\tilde \beta}([a-r_k,b-r_1];\mathbb{R}^n)]\nonumber\\
&\quad+\mathcal{N}[z^{(1)};\mathcal{D}_{\kappa,\alpha,\beta}([a,b];\mathbb{R}^n)]
+\mathcal{N}[z^{(2)};\mathcal{D}_{\kappa,\alpha,\beta}([a,b];\mathbb{R}^n)]
\end{align}
and the constant $c_{x,\varphi,T}$ depends only on $x$, $\varphi$
and $T$.
\end{prop}

\subsection{Integration of delayed controlled paths}
As we have seen in the previous section, the composition with a smooth enough function $\varphi$ transforms a delayed controlled path into a doubly delayed controlled path. We shall see now that the integration with respect to $x$ is acting in the other direction:
\begin{prop}
\label{del:p1.3} For a given $\gamma>1/4$ and $\kappa\leq\gamma$, let
$x$ be a path satisfying Hypothesis \ref{del:h1.1}. Moreover, let
$m\in\widehat{\mathcal{D}}_{\kappa,\hat \alpha,\hat
\beta}([a,b];\mathbb{R}^{1,d})$ such that the increments of $m$
are given by (\ref{del:1.1.2}). Define $z$ by $z_a=\alpha\in
\mathbb{R}$ and
\begin{align}
\label{del:1.3} (\delta z)_{st}&=m^i_s(\delta
x^i)_{st}+(\zeta_s^{(1,i'')})^{i
j}(\mathbf{x}^\mathbf{2}_{st}(r_{i''}))^{j i}\nonumber\\
&\quad+(\zeta_s^{(2,i',j')})^{i j
k}(\mathbf{x}^\mathbf{3}_{st}(r_{j'},r_{i'}))^{k j
i}+(\zeta_s^{(3,i'',j'')})^{i j
k}(\mathbf{x}^\mathbf{3}_{st}(r_{j''}-r_{i''},r_{i''}))^{k j
i}+ \Lambda_{st}(U),\nonumber
\end{align}
for $a\leq s< t \leq b$, $1 \leq i',j'\leq q$, $0 \leq i'',j''\leq
q$, and where $U$ is the increment defined by:
\begin{multline}%\label{del:def-U}
U=\mathcal{R}^i\delta x^i+(\rho^{(i'')})^{i
j}(\mathbf{x}^\mathbf{2}(r_{i''}))^{j i}+\delta
(\zeta^{(2,i',j')})^{i j
k}(\mathbf{x}^\mathbf{3}(r_{j'},r_{i'}))^{k j i} \\
+\delta (\zeta^{(3,i'',j'')})^{i j
k}(\mathbf{x}^\mathbf{3}(r_{j''}-r_{i''},r_{i''}))^{k j i}.
\end{multline}
Finally, set
\begin{equation}
\label{del:1.4}\mathcal{J}(m^idx^i)=\delta z.
\end{equation}
Then,

\smallskip

\noindent
(i)
$\mathcal{J}(m^idx^i)$ coincides with the usual Riemann
integral, whenever $m$ and $x$ are smooth functions.

\smallskip

\noindent
(ii) The path $z$ is well-defined as an element of
$\mathcal{D}_{\kappa,\alpha,\beta}([a,b];\mathbb{R})$, with an initial condition $z_a=\alpha$, $m_a=\beta=\hat \alpha$ and a decomposition of the form
\begin{align*}
(\delta z)_{st}&=m^i\delta x^i+(\zeta^{(1,i'')}_{s})^{i
j}(\mathbf{x}^\mathbf{2}_{st}(r_{i''}))^{ji}+\mathcal{\hat
R}_{st},\\
(\delta m^i)_{st}&=(\zeta^{(1,i'')}_s)^{i j}(\delta
x^j)_{s-r_{i''},t-r_{i''}}+\hat \rho_{st}^{i},
\end{align*}
where the remainders $\mathcal{\hat R}\in
\mathcal{C}_2^{3\kappa}([a,b];\mathbb{R})$ and $\hat
\rho\in\mathcal{C}_2^{2\kappa}([a,b];\mathbb{R}^{1,d})$ are given
by
\begin{align*}
\mathcal{\hat R}_{st}&=(\zeta_s^{(2,i',j')})^{i j
k}(\mathbf{x}^\mathbf{3}_{st}(r_{j'},r_{i'}))^{k j
i}+(\zeta_s^{(3,i'',j'')})^{i j
k}(\mathbf{x}^\mathbf{3}_{st}(r_{j''}-r_{i''},r_{i''}))^{k j
i}+\Lambda_{st}(U),\\
\hat \rho^i_{st}&=(\zeta_s^{(2,i',j')})^{i j
k}(\mathbf{x}^\mathbf{2}_{st}(r_{j'},r_{i'}))^{k
j}+(\zeta_s^{(3,i'',j'')})^{i j
k}(\mathbf{x}^\mathbf{2}_{st}(r_{j''}-r_{i''},r_{i''}))^{k
j}+\mathcal{R}^i_{st},
\end{align*}
respectively.

\smallskip

\noindent
(iii)
The semi-norm of $z$ in
$\mathcal{D}_{\kappa,\alpha,\beta}([a,b];\mathbb{R})$ can be
estimated as
\begin{align}
\label{del:1.5}
&\mathcal{N}[z;\mathcal{D}_{\kappa,\alpha,\beta}([a,b];\mathbb{R})]\nonumber\\
&\leq c_{x,T,\kappa,\gamma}\{1+|\hat
\alpha|_{\mathbb{R}^{1,d}}+(b-a)^{\gamma-\kappa}(|\hat
\alpha|_{\mathbb{R}^{1,d}}+\mathcal{N}[m;\widehat{\mathcal{D}}_{\kappa,\hat
\alpha,\hat \beta}([a,b];\mathbb{R}^{1,d})])\}.
\end{align}
Furthermore, the following bound also holds true:
\begin{align}
\label{del:1.5.1} \|\delta z\|_\kappa\leq
c_{x,T,\kappa,\gamma}(b-a)^{\gamma-\kappa}(|\hat
\alpha|_{\mathbb{R}^{1,d}}+\mathcal{N}[m;\widehat{\mathcal{D}}_{\kappa,\hat
\alpha,\hat \beta}([a,b];\mathbb{R}^{1,d})]).
\end{align}

\smallskip

\noindent
(iv)
The Riemann type sums associated to our generalized integral are of the following form:
\begin{multline}\label{del:1.6}
\mathcal{J}_{st}(m^idx^i)=\lim_{|\Pi_{st}|\to0}\sum_{k'=0}^N\big[m^i_{t_{k'}}(\delta
x^i)_{t_{k'},t_{k'+1}}+(\zeta^{(1,i'')}_{t_{k'}})^{i
j}(\mathbf{x}^\mathbf{2}_{t_{k'},t_{k'+1}}(r_{i''}))^{ji}  \\
+(\zeta^{(2,i',j')}_{t_{k'}})^{i j
k}(\mathbf{x}^\mathbf{3}_{t_{k'},t_{k'+1}}(r_{j'},r_{i'}))^{k j i}
+(\zeta^{(3,i'',j'')}_{t_{k'}})^{i j
k}(\mathbf{x}^\mathbf{3}_{t_{k'},t_{k'+1}}(r_{j''}-r_{i''},r_{i''}))^{k
j i}\big],
\end{multline}
for any $a\leq s<t\leq b$, where the limit is taken over any
partitions $\Pi_{st}=\{t_0=s,\ldots,t_n=t\}$ of $[s,t]$, as the
mesh of the partition goes to zero.
\end{prop}

\begin{proof}
As in the proof of Proposition \ref{p1.2}, we shall mainly focus on the first of these assertions, which justifies our definition of generalized integral. Let us thus suppose for the moment that $x$ is a smooth function and that $m\in
\mathcal{C}_1^\infty(\mathbb{R}^{1,d})$ admits the decomposition
(\ref{del:1.1.2}) with
$\zeta^{(1,i'')}\in\mathcal{C}_1^\infty(\mathbb{R}^{d,d})$,
$\zeta^{(2,i',j')},\zeta^{(3,i'',j'')}\in\mathcal{C}_1^\infty
(\mathbb{R}^{d,d,d})$, for $1\leq i',j'\leq q$ and $0\leq i'',j''\leq q$. We also assume that $\rho^{(i'')}\in\mathcal{C}_2^\infty(\mathbb{R}^{d,d})$ and
$\mathcal{R}\in\mathcal{C}_2^\infty(\mathbb{R}^{1,d})$. Then
$\mathcal{J}(m^i dx^i)$ is well-defined as a Riemann integral, and as for relation (\ref{1.7}), one can write
\begin{equation}
\label{del:1.7}\mathcal{J}(m^idx^i)=m^i\delta x^i + \mathcal{J}(\delta
m^i dx^i).
\end{equation}
Let us now plug the decomposition (\ref{del:1.1.2}) into the
expression (\ref{del:1.7}) in order to obtain
\begin{align}
\label{del:1.8} \mathcal{J}(m^idx^i)&=m^i\delta x^i+(\zeta^{(1,i'')})^{i
j}(\mathbf{x}^\mathbf{2}(r_{i''}))^{j i}+(\zeta^{(2,i',j')})^{i j
k} (\mathbf{x}^\mathbf{3}(r_{j'},r_{i'}))^{k j
i}\nonumber\\
&\quad+(\zeta^{(3,i'',j'')})^{i j k}
(\mathbf{x}^\mathbf{3}(r_{j''}-r_{i''},r_{i''}))^{k j
i}+\mathcal{J}(\mathcal{R}^idx^i),
\end{align}
which can also be written as:
\begin{align}
\label{del:1.8.1} \mathcal{J}(\mathcal{R}^idx^i)&=\mathcal{J}(m^i
dx^i)-m^i\delta x^i-(\zeta^{(1,i'')})^{i
j}(\mathbf{x}^\mathbf{2}(r_{i''}))^{j
i}\nonumber\\
&\quad-(\zeta^{(2,i',j')})^{i j k}
(\mathbf{x}^\mathbf{3}(r_{j'},r_{i'}))^{k j
i}-(\zeta^{(3,i'',j'')})^{i j k}
(\mathbf{x}^\mathbf{3}(r_{j''}-r_{i''},r_{i''}))^{k j i}.
\end{align}
We now apply the operator $\delta$ to both sides of the previous equation in order to get a suitable expression  for a generalization to the rough case:
\begin{align*}
\delta[\mathcal{J}(\mathcal{R}^idx^i)]&
=\delta m^i\delta x^i
+\delta (\zeta^{(1,i'')})^{i
j}(\mathbf{x}^\mathbf{2}(r_{i''}))^{j i}-(\zeta^{(1,i'')})^{i
j}[\delta (x(r_{i''}))^{j}\delta x^i]\\
&\quad+\delta (\zeta^{(2,i',j')})^{i j
k}(\mathbf{x}^\mathbf{3}(r_{j'},r_{i'}))^{k j
i}+\delta (\zeta^{(3,i'',j'')})^{i j
k}(\mathbf{x}^\mathbf{3}(r_{j''}-r_{i''},r_{i''}))^{k j
i}\\
&\quad-(\zeta^{(2,i',j')})^{i j k}[
(\mathbf{x}^\mathbf{2}(r_{j'},r_{i'}))^{k j}\delta x^i+\delta (x(r_{i'}+r_{j'}))^k(\mathbf{x}^\mathbf{2}(r_{i'}))^{ji}]\\
&\quad-(\zeta^{(3,i'',j'')})^{i j k}[
(\mathbf{x}^\mathbf{2}(r_{j''}-r_{i''},r_{i''}))^{k j}\delta
x^i+\delta (x(r_{j''}))^k(\mathbf{x}^\mathbf{2}(r_{i''}))^{j i}].
\end{align*}
Taking now into account the decomposition (\ref{del:1.1.2}) for
$m$ and (\ref{del:1.1.3}) for $\zeta^{(1,i'')}$, we obtain:
\begin{multline}\label{del:1.9}
\delta[\mathcal{J}(\mathcal{R}^idx^i)]= \mathcal{R}^i\delta
x^i+(\rho^{(i'')})^{i j}(\mathbf{x}^\mathbf{2}(r_{i''}))^{j
i}+\delta (\zeta^{(2,i',j')})^{i j
k}(\mathbf{x}^\mathbf{3}(r_{j'},r_{i'}))^{k j i}  \\
+\delta (\zeta^{(3,i'',j'')})^{i j
k}(\mathbf{x}^\mathbf{3}(r_{j''}-r_{i''},r_{i''}))^{k j i}=U,
\end{multline}
where $U$ is defined by equation (\ref{del:1.3}).
If we now only assume that $U$ is an element of $\cac_3^\mu$ with $\mu>1$, one gets that $U\in{\rm Dom}(\Lambda)$, which yields the decomposition (\ref{del:1.4}) for $\cj(m^i dx^i)$. The remainder of the proof is then just made of tedious elementary estimations for the regularity of all the terms involved in the decomposition (\ref{del:1.4}). Finally, observe that the Riemann sum limit is obtained by applying Corollary \ref{c2.1}.

\end{proof}

As in Corollary \ref{c1.1}, the previous Proposition can be extended easily to the multidimensional case:
\begin{cor}
Let $m\in \widehat{\mathcal{D}}_{\kappa,\hat \alpha,\hat
\beta}([a,b];\mathbb{R}^{n,d})$ be a doubly delayed controlled path, and define another path $z$ by $\delta z^i=\mathcal{J}(m^{ij}dx^j)$ for $i=1,\ldots, n$. Then $z$ is well-defined as an element of $\mathcal{D}_{\kappa,\alpha,\beta}([a,b];\mathbb{R}^n)$, and the following bound holds true:
\begin{align}
\label{del:1.10}
&\mathcal{N}[z;\mathcal{D}_{\kappa,\alpha,\beta}([a,b];\mathbb{R}^n)]\nonumber\\
&\leq c_{x,T,\kappa,\gamma}\{1+|\hat
\alpha|_{\mathbb{R}^{n,d}}+(b-a)^{\gamma-\kappa}(|\hat
\alpha|_{\mathbb{R}^{n,d}}+\mathcal{N}[m;\widehat{\mathcal{D}}_{\kappa,\hat
\alpha,\hat \beta}([a,b];\mathbb{R}^{n,d})])\}.
\end{align}
\end{cor}
Another useful feature of our generalized integral is a continuity property with respect to the integrand $m$ in $\mathcal{J}(z\, dx)$, whose proof is omitted again for sake of conciseness:
\begin{prop}
Let $m^{(1)},m^{(2)}\in \widehat{\mathcal{D}}_{\kappa,\hat \alpha,\hat
\beta}([a,b];\mathbb{R}^{n,d})$ be two doubly delayed controlled paths, and
define $z^{(1)},z^{(2)}\in
\mathcal{D}_{\kappa,\alpha,\beta}([a,b];\mathbb{R}^n)$ by
$\delta z^{(l)}=\mathcal{J}(m^{(l)}\, dx)$, for $l=1,2$. Assume moreover that the paths $\zeta^{(1,i'';1)}$ and $\zeta^{(1,i'';2)}$ in the respective decompositions (\ref{del:1.1.2}) of $z^{(1)}$ and $z^{(2)}$ satisfy $\zeta^{(1,i'';1)}_a=\zeta^{(1,i'';2)}_a$, and that $b-a<1$. Then we have:
\begin{multline}\label{del:1.11}
\mathcal{N}[z^{(1)}-z^{(2)};\mathcal{D}_{\kappa,0,0}([a,b];\mathbb{R}^n)]  \\
\leq
c_{x,T,\kappa,\gamma}(b-a)^{\kappa\wedge(\gamma-\kappa)}\mathcal{N}[m^{(1)}-m^{(2)};\widehat{\mathcal{D}}_{\kappa,0,0}([a,b];\mathbb{R}^{n,d})].
\end{multline}
\end{prop}

\subsection{Rough delay differential equations}
\label{del:s2}

We shall now turn to the main goal of this section, namely the resolution of equation (\ref{def-delay-eq}), which can be written now as:
\begin{equation}\label{def-delay-eq-intg}
(\der y)_{st}=\cj_{st}\lp \si(y,\mathfrak{s}(y))\, dx\rp, \quad s,t\in[0,T],
\end{equation}
with initial condition $y_t=\xi_t$ for $t\in[-r,0]$, where we recall the convention (\ref{def-s(y)}) for $\mathfrak{s}$, and where the integral $\cj$ has to be interpreted according to Proposition \ref{del:p1.3}.

\smallskip

Before stating our main result in this direction, let us introduce a natural map, called $\Gamma$, associated to our delay equation. It is defined by $\Gamma(z,\tilde z):= \hat z$, where (recalling the notation $T_\si$ introduced at Proposition \ref{del:p1.1}) $\delta \hat z$ is given as $\delta \hat z=\cj(T_\sigma(z,\tilde z)\, dx)$, on the following spaces:
$$
\Gamma:\mathcal{D}_{\kappa,\alpha,\beta}([a,b];\mathbb{R}^n)\times\mathcal{D}_{\kappa,\tilde\alpha,\tilde\beta}([a-r_{q},b-r_1];\mathbb{R}^n)
\rightarrow \mathcal{D}_{\kappa,\alpha,\beta}([a,b];\mathbb{R}^n)
$$
for $0\leq a < b \leq T$. In the previous definition, $\al,\tilde\al,\beta$ and $\tilde\beta$ stand for some initial conditions, with the additional compatibility condition $\beta=\sigma(\alpha,\tilde z_{a-r_1},\ldots,\tilde z_{a-r_{q-1}},\tilde z_{a-r_{q}})$, which shall be satisfied in our delay equation context. Notice also that, from now on, we shall use the
convention that $z_s=\tilde z_s=\hat z_s=\xi_s$ for $s\in[-r_q,0]$. Since we have assumed $\xi\in\cac_1^{3\gamma}$, this means in particular that, on $[-r_q,0]$, the paths $z,\tilde z,\hat z$ are still controlled processes, whose degenerate decomposition is only given by a remainder term. This allows to complete easily the definition of $\Gamma$ on intervals of the form $[a,b]$ with $a<r_q$.

\begin{rem}
As in \cite{NNT07}, we could have handled the case of a controlled initial condition $\xi$. We did not consider this possibility here for sake of conciseness.
\end{rem}

Let us gather now some useful relations concerning the operator $\Gamma$ we have just defined: first of all, by putting together inequalities (\ref{del:1.10}) and (\ref{del:1.1.1}), one obtains
\begin{multline}\label{del:2.4}
\mathcal{N}[\Gamma(z,\tilde z);\mathcal{D}_{\kappa,\alpha,\beta}([a,b];\mathbb{R}^n)]\leq c_{\sigma,x,T,\kappa,\gamma}\big(1+\mathcal{N}^3[\tilde
z;\mathcal{D}_{\kappa,\tilde \alpha,\tilde \beta}([a-r_{q},b-r_1];\mathbb{R}^n)]\big)  \\
\times\big(1+(b-a)^{\gamma-\kappa}\mathcal{N}^3[z;\mathcal{D}_{\kappa,\alpha,\beta}([a,b];\mathbb{R}^n)]\big),
\end{multline}
which means that the semi-norm of the mapping
$\Gamma$ is cubically bounded in terms of the semi-norm of $z$ and
$\tilde z$.

\smallskip

Furthermore, let $z^{(1)},z^{(2)}\in
\mathcal{D}_{\kappa,\alpha,\beta}([a,b];\mathbb{R}^n)$ and $\tilde
z \in \mathcal{D}_{\kappa,\tilde \alpha,\tilde
\beta}([a-r_{q},b-r_1];\mathbb{R}^n)$. Then, if $b-a<1$, applying successively the inequalities (\ref{del:1.11}) and  (\ref{del:lips}), we get that
\begin{align}
\label{del:2.6}
&\mathcal{N}[\Gamma(z^{(1)},\tilde z)-\Gamma(z^{(2)},\tilde z);\mathcal{D}_{\kappa,0,0}([a,b];\mathbb{R}^n)]\nonumber\\
&\leq \tilde
c_{x,\sigma,T,\kappa,\gamma}(1+C(z^{(1)},z^{(2)},\tilde
z))^3(b-a)^{\kappa\wedge(\gamma-\kappa)}\mathcal{N}[z^{(1)}-z^{(2)};\mathcal{D}_{\kappa,0,0}([a,b];\mathbb{R}^n)],
\end{align}
where $C(z^{(1)},z^{(2)},\tilde z)$ is defined at (\ref{del:cons-lips}).
Therefore, for fixed $\tilde z$ the mappings $\Gamma(\cdot,\tilde
z)$ are locally Lipschitz continuous with respect to the semi-norm
$\mathcal{N}[\cdot;\mathcal{D}_{\kappa,0,0}([a,b];\mathbb{R}^n)]$.

\smallskip

With these preliminary results in hand, we can now prove our main
theorem on delay equations:
\begin{thm}
\label{del:t2.1}Let $x$ be a path satisfying Hypotheses \ref{del:h1.1} and
\ref{del:h1.2}, let $\sigma:\mathbb{R}^{n,q+1}\rightarrow
\mathbb{R}^{n,d}$ be a $C^4_b$-function and let $\xi\in
\mathcal{C}_1^{3\gamma}([-r,0];\mathbb{R}^n)$. Then

\smallskip

\noindent
(i) Equation (\ref{def-delay-eq-intg}) admits a unique solution
$y$ in
$\mathcal{D}_{\kappa,\xi_0,\sigma_r(\xi_0)}([0,T];\mathbb{R}^n)$,
for any $\kappa<\gamma$ such that $3\kappa+\gamma>1$ and any
$T>0$, where
$\sigma_r(\xi_0)\triangleq\sigma(\xi_0,\xi_{-r_1},\ldots,\xi_{-r_{q}})$.

\smallskip

\noindent
(ii)The map
$$
\Big(\xi,x,\{\mathbf{x}^\mathbf{2}(r_{i''})\}_{0\leq i''\leq
q},\{\mathbf{x}^\mathbf{3}(r_{j'},r_{i'})\}_{1\leq i',j'\leq
q},\{\mathbf{x}^\mathbf{3}(r_{j''}-r_{i''},r_{i''})\}_{0\leq
i'',j''\leq q}\Big)\mapsto y
$$
is locally Lipschitz continuous from
$$
{C}_1^{3\gamma}([-r,0];\mathbb{R}^n)\times\mathcal{C}_1^\gamma([0,T];\mathbb{R}^d)\times
\big(\mathcal{C}_2^{2\gamma}([0,T];\mathbb{R}^{d,d})\big)^{q+1}
\times
\big(\mathcal{C}_2^{3\gamma}([0,T];\mathbb{R}^{d,d,d})\big)^{2q^2+2q+1}
$$
to $\mathcal{C}_1^\kappa([0,T];\mathbb{R}^n$, in the following sense:
let $\tilde x$ be another driving rough path with corresponding
delayed L\'evy area and doubly delayed volume element
$\mathbf{\tilde x}^\mathbf{2}(v)$, $\mathbf{\tilde
x}^\mathbf{3}(v'',v)$, respectively, where $v''=v' \text{ or
}v'-v$, with $v,v'\in\{r_q,\ldots,r_0\}$, and $\tilde \xi$
another initial condition. Then, for every $M>0$, there exists a
constant $K_M>0$ such that the upper bound
\begin{align*}
&\mathcal{N}[y-\tilde y;\mathcal{C}_1^\kappa([0,T];\mathbb{R}^n)]\leq K_M
\Big\{\mathcal{N}[\xi-\tilde
\xi;\mathcal{C}_1^{3\gamma}([0,T];\mathbb{R}^n)]+\mathcal{N}[x-\tilde
x;\mathcal{C}_1^\gamma([0,T];\mathbb{R}^n)]\\
&
\quad+\sum_{i''=0}^{q}\mathcal{N}[\mathbf{x}^\mathbf{2}(r_{i''})-\mathbf{\tilde
x}^\mathbf{2}(r_{i''});\mathcal{C}_2^{2\gamma}([0,T];\mathbb{R}^{d,d})]\\
&
\quad+\sum_{i',j'=1}^{q}\mathcal{N}[\mathbf{x}^\mathbf{3}(r_{j'},r_{i'})-\mathbf{\tilde
x}^\mathbf{3}(r_{j'},r_{i'});\mathcal{C}_2^{3\gamma}([0,T];\mathbb{R}^{d,d,d})]\\
&
\quad+\sum_{i'',j''=0}^{q}\mathcal{N}[\mathbf{x}^\mathbf{3}(r_{j''}-r_{i''},r_{i''})-\mathbf{\tilde
x}^\mathbf{3}(r_{j''}-r_{i''},r_{i''});\mathcal{C}_2^{3\gamma}([0,T];\mathbb{R}^{d,d,d})]\Big\}.
\end{align*}
holds for all tuples
\begin{align*}
&\big(\xi,x,\{\mathbf{x}^\mathbf{2}(r_{i''})\}_{0\leq i''\leq
q},\{\mathbf{x}^\mathbf{3}(r_{j'},r_{i'})\}_{1\leq i',j'\leq
q},\{\mathbf{x}^\mathbf{3}(r_{j''}-r_{i''},r_{i''})\}_{0\leq
i'',j''\leq q}\big),\\
&\big(\tilde \xi,\tilde x,\{\mathbf{\tilde
x}^\mathbf{2}(r_{i''})\}_{0\leq i''\leq q},\{\mathbf{\tilde
x}^\mathbf{3}(r_{j'},r_{i'})\}_{1\leq i',j'\leq
q},\{\mathbf{\tilde
x}^\mathbf{3}(r_{j''}-r_{i''},r_{i''})\}_{0\leq i'',j''\leq
q}\big),
\end{align*}
satisfying the boundedness condition
\begin{align*}
&\mathcal{N}[\xi;\mathcal{C}_1^{3\gamma}([0,T];\mathbb{R}^n)]+\mathcal{N}[\tilde
\xi;\mathcal{C}_1^{3\gamma}([0,T];\mathbb{R}^n)]
+\mathcal{N}[x;\mathcal{C}_1^\gamma([0,T];\mathbb{R}^n)]+\mathcal{N}[\tilde
x;\mathcal{C}_1^\gamma([0,T];\mathbb{R}^n)]\\
&
+\sum_{i''=0}^{q}\mathcal{N}[\mathbf{x}^\mathbf{2}(r_{i''});\mathcal{C}_2^{2\gamma}([0,T];\mathbb{R}^{d,d})]+\sum_{i''=0}^{q}\mathcal{N}[\mathbf{\tilde
x}^\mathbf{2}(r_{i''});\mathcal{C}_2^{2\gamma}([0,T];\mathbb{R}^{d,d})]\\
&
+\sum_{i',j'=1}^{q}\mathcal{N}[\mathbf{x}^\mathbf{3}(r_{j'},r_{i'});\mathcal{C}_2^{3\gamma}([0,T];\mathbb{R}^{d,d,d})]+\sum_{i',j'=1}^{q}\mathcal{N}[\mathbf{\tilde
x}^\mathbf{3}(r_{j'},r_{i'});\mathcal{C}_2^{3\gamma}([0,T];\mathbb{R}^{d,d,d})]\\
&+\sum_{i'',j''=0}^{q}\mathcal{N}[\mathbf{x}^\mathbf{3}(r_{j''}-r_{i''},r_{i''});\mathcal{C}_2^{3\gamma}([0,T];\mathbb{R}^{d,d,d})]\\
&+\sum_{i'',j''=0}^{q}\mathcal{N}[\mathbf{\tilde
x}^\mathbf{3}(r_{j''}-r_{i''},r_{i''});\mathcal{C}_2^{3\gamma}([0,T];\mathbb{R}^{d,d,d})]\leq
M.
\end{align*}
\end{thm}

\begin{proof}
With the previous algebraic structures and computations in hand, the proof of this theorem follows the lines of \cite[Theorem 4.2]{NNT07}. We shall give some hints for the proof of the existence-uniqueness result, which is based on a fixed point argument for the map $\Gamma$ defined above, for sake of completeness.

\smallskip

Without loss of generality suppose that $T=Nr_1$, where we recall that $r_1$ is the smallest delay in
(\ref{def-delay-eq-intg}). We shall
construct the solution of our delay equation by
induction over the intervals $[0,r_1]$, $[0,2r_1]$, $\ldots$,
$[0,Nr_1]$.

\smallskip

Let us first show that equation
(\ref{def-delay-eq-intg}) has a solution on the interval $[0,r_1]$.
To this purpose, define
$$
\tilde \tau_1=\Big(\frac{1}{c_1
M_1^2}-\frac{1}{M_1^3}\Big)^\frac{1}{\gamma-\kappa}\wedge
r_1\wedge 1
$$
where
$M_1>c_1=c_{\kappa,\gamma,\sigma,T}(1+\mathcal{N}^3[\xi;\mathcal{C}_1^{3\gamma}([-r_{q},0];\mathbb{R}^n)])$.
In addition, choose $\tau_1 \in [0,\tilde \tau_1]$ and $N_1\in
\mathbb{N}$ such that $N_1\tau_1=r_1$, and define
$$
I_{k',1}=[(k'-1)\tau_1,k'\tau_1],\qquad k'=1,\ldots,N_1.
$$
Finally, consider the following map:
let
$\Gamma_{1,1}:\mathcal{D}_{\kappa,\xi_0,\sigma_r(\xi_0)}(I_{1,1};\mathbb{R}^n)\rightarrow
\mathcal{D}_{\kappa,\xi_0,\sigma_r(\xi_0)}(I_{1,1};\mathbb{R}^n)$
be given by $\hat z=\Gamma_{1,1}(z)$, where
$$
(\delta  \hat z^i)_{st}=\mathcal{J}_{st}(T_\sigma^{i
j}(z,\xi)dx^j)
$$
for $0\leq s < t \leq \tau_1$.
Notice then that if $z^{(1,1)}$ is a fixed point of the map
$\Gamma_{1,1}$, then $z^{(1,1)}$ solves equation
(\ref{def-delay-eq-intg}) on the interval $I_{1,1}$. Therefore, we
shall prove that such a fixed point exists.

\smallskip

First, owing to (\ref{del:2.4}) we get the estimate
\begin{align}
\label{del:2.7}
\mathcal{N}[\Gamma_{1,1}(z);\mathcal{D}_{\kappa,\xi_0,\sigma_r(\xi_0)}(I_{1,1};\mathbb{R}^n)]\leq
c_1\big(1+\tau_1^{\gamma-\kappa}\mathcal{N}^3[z;\mathcal{D}_{\kappa,\xi_0,\sigma_r(\xi_0)}(I_{1,1};\mathbb{R}^n)]\big).
\end{align}
Therefore, thanks to our previous choice of $\tau_1$, we obtain that the ball
\begin{equation}
\label{del:2.8} B_{M_1}=\{z\in
\mathcal{D}_{\kappa,\xi_0,\sigma_r(\xi_0)}(I_{1,1};\mathbb{R}^n);\;\mathcal{N}[z;\mathcal{D}_{\kappa,\xi_0,\sigma_r(\xi_0)}(I_{1,1};\mathbb{R}^n)]\leq
M_{1}\}
\end{equation}
is invariant under $\Gamma_{1,1}$.
On the other hand, by changing $\tau_1$ to a smaller value (and
then $N_1$ accordingly) if necessary, observe that $\Gamma_{1,1}$
also is a contraction on $B_{M_1}$, see (\ref{del:2.6}). Thus,
applying the Fixed Point Theorem, it is easily shown that there exists a
unique solution $z^{(1,1)}$ to equation (\ref{def-delay-eq-intg}) on
the interval $I_{1,1}$.

\smallskip

If $\tau_1=r_1$, we have thus obtained the existence and uniqueness of a solution in the interval $[0,r_1]$.
Otherwise, define the map
$$
\Gamma_{2,1}:\mathcal{D}_{\kappa,z_{r_1}^{(1,1)},\sigma_r(
z_{r_1}^{(1,1)})}(I_{2,1};\mathbb{R}^n)\longrightarrow
\mathcal{D}_{\kappa,z_{r_1}^{(1,1)},\sigma_r(
z_{r_1}^{(1,1)})}(I_{2,1};\mathbb{R}^n)
$$ given by $\hat
z=\Gamma_{2,1}(z)$, where $\sigma_r(
z_{r_1}^{(1,1)})\triangleq\sigma(z_{r_1}^{(1,1)},\xi_0,\ldots,\xi_{-r_{q-1}})$
and
$$
(\delta  \hat z^i)_{st}=\mathcal{J}_{st}(T_\sigma^{i
j}(z,\xi)dx^j)
$$
for $\tau_1\leq s < t \leq 2\tau_1$. Since $\tau_1<r_1$, the following upper bound still holds true:
\begin{align}
\label{del:2.9}
\mathcal{N}[\Gamma_{2,1}(z);\mathcal{D}_{\kappa,z_{r_1}^{(1,1)},\sigma_r(z_{r_1}^{(1,1)})}(I_{2,1};\mathbb{R}^n)]\leq
c_1\big(1+\tau_1^{\gamma-\kappa}\mathcal{N}^3[z;\mathcal{D}_{\kappa,z_{r_1}^{(1,1)},\sigma_r(z_{r_1}^{(1,1)})}(I_{2,1};\mathbb{R}^n)]\big)
\end{align}
and we obtain, resorting to the same fixed point argument as above, the
existence of a unique solution $z^{(2,1)}$ to equation
(\ref{def-delay-eq-intg}) on the interval $I_{2,1}$.
Repeating this step as often as necessary, which is possible since
the estimates on the norms of the mappings $\Gamma_{l,1}$,
$l=1,\ldots,N_1$ are of the same type as (\ref{del:2.4}), that is, the
constant $c_1$ in (\ref{del:2.9}) does not change according to the iteration step, we obtain that
$z=\sum_{l=1}^{N_1}z^{(l,1)}\Ind_{I_{l,1}}$ is the unique solution
to the equation (\ref{def-delay-eq-intg}) on the interval $[0,r_1]$.

\smallskip

The patching of solutions defined on different intervals of the form $I_{l,k}$ is then a slight elaboration of the computations corresponding to \cite[Theorem 4.2]{NNT07}, and this step is left to the reader. The continuity of the Itô map is follows also the steps of \cite{NNT07}, except for the huge number of terms we have to deal with in the current situation. We prefer to omit this step for sake of conciseness.

\end{proof}

\section{Application to the fractional Brownian motion}\label{sec:fbm}

All the previous constructions rely on the specific assumptions
that we have made on the process $x$. In this section, we
prove how our results can be applied to the fractional Brownian
motion. More specifically, we first recall some basic definitions about fBm, and then define the delayed Lévy area $\mathbf{B^2}$. We shall then turn to the definition of the volume $\mathbf{B^3}$, which is the main difficulty in order to go from the case $H>1/3$ treated in \cite{NNT07} to our rougher situation.

\subsection{Basic facts on fractional Brownian motion}
Recall that a $d$-dimensional fBm with Hurst parameter $H\in(0,1)$
defined on the real line is a centered Gaussian process
$$
B=\{B_t=(B_t^1,\ldots,B_t^d);\;t\in\mathbb{R}\},
$$
where $B^1,\ldots,B^d$ are $d$ independent 1-dimensional fBm, that
is, each $B^i$ is a centered Gaussian process with continuous
sample paths and covariance function
\begin{equation}
\label{cov-func}
R_H(t,s)=\mathbb{E}(B_t^iB_s^i)=\frac{1}{2}(|t|^{2H}+|s|^{2H}-|t-s|^{2H}),
\end{equation}
for all $i\in\{1,\ldots,d\}$. In the sequel, all the random
variables we deal with are defined on a complete probability space
$(\Omega,\cf,\mathbb{P})$, and we assume that $\cf$ is generated
by the random variables $(B_t;\, t\in \mathbb{R})$. The fBm
verifies the following two important properties:

\smallskip

\begin{itemize}
\item Scaling property: for any $c>0$, $B^{(c)}=c^HB_{\cdot/c}$ is
a fBm,

\smallskip

\item Stationarity property: for any $h\in \mathbb{R}$,
$B_{\cdot+h}-B_h$ is a fBm.
\end{itemize}

\smallskip

\noindent
Notice that we work with a fBm indexed by $\mathbb{R}$ for sake
of simplicity as in \cite{NNT07}, since this  allows some more
elegant calculations for the definitions of the double delayed
L\'evy area and volume, respectively. Furthermore, since the case $H>1/2$ or the Brownian case $H=1/2$ are less demanding than the rougher case, we shall mainly focus in this section on the range of parameter $H<1/2$.

\subsubsection{Gaussian structure of $B$}
Let us give a few facts about the Gaussian structure of fractional
Brownian motion, following
Chapter 5 of \cite{Nu06}. All the considerations in this direction will concern a  1-dimensional fBm $B$, which will be enough for our applications.

\smallskip

Let $\mathcal{E}$ be the set of step-functions on
$\mathbb{R}$ with values in $\mathbb{R}$. Consider the Hilbert
space $\mathcal{H}$ defined as the closure of $\mathcal{E}$ with
respect to the scalar product induced by
\begin{align*}
\left\langle \Ind_{[t,t']}, \Ind_{[s,s']}
\right\rangle_{\mathcal{H}}=
R_H(t',s')-R_H(t',s)-R_H(t,s')+R_H(t,s),
\end{align*}
for any $-\infty<s<s'<+\infty$ and $-\infty<t<t'<+\infty$, and
where $R_H(t,s)$ is given by (\ref{cov-func}). The mapping
$$
\Ind_{[t,t']} \mapsto  B_{t'}  - B_t
$$
can be extended to an isometry between $\mathcal{H}$ and the
Gaussian space $H_{1}(B)$ associated with $B$. We denote this
isometry by $\varphi \mapsto B(\varphi)$.

\smallskip

The spaces $\ch$ and $H_1(B)$ can be characterized more precisely in the following way: first, we notice that a 1-dimensional fBm defined on the real
line, with $H\neq1/2$, has the following integral representation
in terms of a Wiener process $W$ defined on $\mathbb{R}$ (See \cite[Proposition 7.2.6]{ST94} for details):
\begin{align}
\label{moving_average_rep}
B_t=\frac{1}{C_1(H)}\int_{\mathbb{R}}\left[(t-s)_{+}^{H-1/2}-(-s)_{+}^{H-1/2}\right]dW_s,\quad
t\in\mathbb{R},
\end{align}
where
\begin{equation}\label{eq:61a}
C_1(H)=\left(\int_0^\infty\left[(1+s)^{H-1/2}-s^{H-1/2}\right]^2ds+\frac{1}{2H}\right)^{1/2},
\end{equation}
and where $a_+$ stands for the positive part of a real number $a$, namely $a_+=\Ind_{\R_+}(a)\, a$. Using the representation (\ref{moving_average_rep}), the authors
in \cite{PT00} define the following stochastic integral of a
deterministic function with respect to a 1-dimensional fBm $B$:
\begin{align*}
\int_\mathbb{R}f(u)dB_u=\frac{\Gamma\left(H+1/2\right)}{C_1(H)}\left\{
\begin{array}{ll}
\int_{\mathbb{R}}\left(\mathcal{D}_-^{1/2-H}f\right)(u)dW_u,\quad &H<1/2,\\
\\
\int_{\mathbb{R}}\left(\mathcal{I}_-^{H-1/2}f\right)(u)dW_u,\quad
&H>1/2,
\end{array}
\right.
\end{align*}
provided that the stochastic integral with respect to the Wiener
process $W$ makes sense, and where
\begin{align}
&\left(\mathcal{D}_-^{\alpha}f\right)(u)=\frac{\alpha}{\Gamma(1-\alpha)}\int_0^\infty\frac{f(r)-f(u+r)}{r^{1+\alpha}}dr,\label{deriv-frac}\\
&\left(\mathcal{I}_-^{\alpha}f\right)(u)=\frac{1}{\Gamma(\alpha)}\int_u^\infty\frac{f(r)}{(r-u)^{1-\alpha}}dr\label{int-frac},
\end{align}
for $0<\alpha<1$. The expressions (\ref{deriv-frac}) and (\ref{int-frac}) are respectively called
right-sided fractional derivative  and right-sided fractional integral on the whole real
line. We remark that, in general,
\begin{equation*}
\left(\mathcal{D}_-^{\alpha}f\right)(u)\equiv\underset{\varepsilon\rightarrow
0}{\lim}\frac{\alpha}{\Gamma(1-\alpha)}\int_\varepsilon^\infty\frac{f(r)-f(u+r)}{r^{1+\alpha}}dr.
\end{equation*}
We also notice that
\begin{equation}
\label{rel-frac-deriv-int}
\left(\mathcal{I}_-^{\alpha}(\left(\mathcal{D}_-^{\alpha}f\right)\right)(u)=\left(\mathcal{D}_-^{\alpha}\left(\mathcal{I}_-^{\alpha}f\right)\right)(u)
=f(u).
\end{equation}

\smallskip

When $f$ is a function defined on an interval $[a,b]$ with $-\infty<a<b<\infty$, extend $f$ by setting $f^\star=f\Ind_{[a,b]}$. Define then
\begin{align}
&\left(\mathcal{D}_-^{\alpha}f^\star\right)(u)=\left(\mathcal{D}_{-b}^{\alpha}f\right)(u)=\frac{f(u)}{\Gamma(1-\alpha)(b-u)^\alpha}+\frac{\alpha}{\Gamma(1-\alpha)}\int_u^b\frac{f(u)-f(r)}{(r-u)^{1+\alpha}}dr,\label{deriv-frac-int}\\
&\left(\mathcal{I}_-^{\alpha}f^\star\right)(u)=\left(\mathcal{I}_{-b}^{\alpha}f\right)(u)=\frac{1}{\Gamma(\alpha)}\int_u^b\frac{f(r)}{(r-u)^{1-\alpha}}dr\label{int-frac-int},
\end{align}
for $0<\alpha<1$, $a<u<b$. The expressions (\ref{deriv-frac-int}) and (\ref{int-frac-int}) are respectively called
right-sided fractional derivative  and right-sided fractional integral on the interval $[a,b]$. In this context, as in the case of the whole line (see \cite{SKM93} for details and also \cite{Za98}), the following relation holds true:
\begin{equation*}
\left(\mathcal{D}_{-b}^{\alpha}f\right)(u)\equiv\frac{f(u)}{\Gamma(1-\alpha)(b-u)^\alpha}+\underset{\varepsilon\rightarrow
0}{\lim}\frac{\alpha}{\Gamma(1-\alpha)}\int_{u+\varepsilon}^b\frac{f(u)-f(r)}{(r-u)^{1+\alpha}}dr.
\end{equation*}

\smallskip

With these notations in hand, it is proved in \cite{PT00} that the operator
\begin{align*}
(\mathcal{K}f)(u)\equiv
\frac{\Gamma\left(H+1/2\right)}{C_1(H)}\left\{
\begin{array}{ll}
\left(\mathcal{D}_-^{1/2-H}f\right)(u),\quad &H<1/2,\\
\\
\left(\mathcal{I}_-^{H-1/2}f\right)(u),\quad &H>1/2,
\end{array}
\right.
\end{align*}
is an isometry between $\mathcal{H}$ and a closed subspace of
$L^2(\mathbb{R})$. In fact,
$$
\langle \phi,\psi \rangle_\mathcal{H}=\langle
\mathcal{K}\phi,\mathcal{K}\psi \rangle_{L^2(\mathbb{R})},
$$
for all $\phi,\psi\in \mathcal{H}$. This also allows to write $B(\varphi)$ as $W(\mathcal{K}\varphi)$ for any $\varphi\in\ch$, where $W(\mathcal{K}\varphi)$ has to be interpreted as a Wiener integral with respect to the Gaussian measure $W$. In particular, we have:
\begin{equation}\label{eq:65b}
\mathbb{E}[|B(\varphi)|^2]=\|\varphi\|_{\ch}
=\|\mathcal{K}\varphi\|_{L^2(\mathbb{R})}.
\end{equation}

\subsubsection{Malliavin calculus with respect to the fBm $B$}
Let $\mathcal{S}$ be the
set of smooth cylindrical random variables of the form
$$
F = f(B(\varphi_1), \ldots, B(\varphi_k)), \qquad \varphi_i \in
\mathcal{H}, \quad i\in\{1, \ldots, k\},
$$
where $f\in C^{\infty}(\mathbb{R}^{d,k},\mathbb{R})$ is bounded
with bounded derivatives. The derivative operator $D$ of a smooth
cylindrical random variable of the above form is defined as the
$\mathcal{H}$-valued random variable
$$
DF= \sum_{i=1}^k \frac{\partial f}{\partial x_i}
(B(\varphi_1),\ldots, B(\varphi_k)) \varphi_i.
$$
This operator is
closable from $L^p(\Omega)$ into $L^{p}(\Omega; \mathcal{H})$. As
usual, $\mathbb{D}^{1,2}$ denotes the closure of the set of smooth
random variables with respect to the norm
$$
\|F\|_{1,2}^2 = \mathbb{E}|F|^2+\mathbb{E}\|DF\|_{\mathcal{H}}^2.
$$
In particular, considering a $d$-dimensional fBm
$(B^1,\ldots,B^d)$, if $D^{B^i}F$ denotes the Malliavin derivative
of $F\in\mathbb{D}^{1,2}_{B^i}$ with respect to $B^i$, where
$\mathbb{D}^{1,2}_{B^i}$ denotes the corresponding Sobolev space,
we have $D^{B^i}B^j_{t}=\delta_{i,j}{\Ind}_{(-\infty,t]}$ for
$i,j=1, \ldots, d$, where $\delta_{i,j}$ denotes the Kronecker
symbol.

\smallskip

The divergence operator $I$ is the adjoint of the derivative
operator. If a random variable $\phi \in L^{2}(\Omega;
\mathcal{H})$ belongs to ${\rm dom}(I)$, the domain of the
divergence operator, then $I(\phi)$ is defined by the duality
relationship
\begin{equation}
\label{duality}
 \mathbb{E}(FI(\phi))= \mathbb{E} \langle D F, \phi \rangle_{\mathcal{H}},
\end{equation}
for every $F \in \mathbb{D}^{1,2}$. In additon, let us recall two
useful properties verified by $D$ and $I$:
\begin{itemize}
\item If $\phi \in {\rm dom}(I)$ and $F \in \mathbb{D}^{1,2}$ such
that $F\phi \in L^2(\Omega; \mathcal{H})$, then we have the
following integration by parts formula:
\begin{equation}
\label{parts-form} I(F\phi)=FI(\phi)- \langle DF, \phi
\rangle_{\mathcal{H}}.
\end{equation}
\item If $\phi\in \mathbb{D}^{1,2}(\mathcal{H})$, $D_r \phi \in
{\rm dom}(I)$ for all $r \in\mathbb{R}$ and $ \{I(D_r \phi)\}_{r
\in \mathbb{R}}$ is an element of $L^{2}(\Omega; \mathcal{H})$,
then
\begin{equation}
\label{deriv-int} D_rI( \phi)= \phi_r + I(D_r \phi).
\end{equation}
\end{itemize}

\smallskip

One can relate the Malliavin derivatives with respect to $B$ and $W$ through the operator $\mathcal{K}$ defined above. Indeed, relation (\ref{rel-frac-deriv-int}) shows that  $\mathcal{K}$ is invertible. This allows to state, as in the case of a
1-dimensional fBm $B$ in an interval (see for example
\cite[Section 5.2]{Nu06} and also \cite{ALN01}), the following
relations for the Malliavin derivative and divergence operators
with respect to the processes $B$ and $W$:
\begin{itemize}
\item[(i)] For any $F \in \mathbb{D}^{1,2}_W=\mathbb{D}^{1,2}$, we have:
$$
\mathcal{K}DF = D^WF,
$$
where $D^W$ denotes the derivative operator with respect to the
process W, and $\mathbb{D}^{1,2}_W$ the corresponding Sobolev
space.

\item[(ii)] ${\rm Dom} (I) = \mathcal{K}^{-1}({\rm Dom} (I^W))$,
and for any $\mathcal{H}$-valued random variable $u$ in ${\rm
Dom}(I)$ we have $I(u) = I^W(\mathcal{K}u)$, where $I^W$ denotes
the divergence operator with respect to the process $W$.
\end{itemize}

\smallskip

\noindent
In addition, we have
$\mathbb{D}^{1,2}(\mathcal{H})=(\mathcal{K}^{-1})(\mathbb{L}^{1,2})$,
where $\mathbb{L}^{1,2}=\mathbb{D}^{1,2}(L^2(\mathbb{R}))$, and
this space is included in ${\rm dom}(I^W)$. Making use of
the notations $I^W(\phi)=\int_{R}\phi_u\;dW_u$ for any
$\phi\in{\rm dom}( I^{W})$, and $I(\phi)=\int_{R}\phi_u\;dB_u$ for
any $\phi\in{\rm dom}(I)$, we can write:
$$
\int_{\mathbb{R}}\phi_u\;dB_u=\int_{\mathbb{R}}(\mathcal{K}\phi)(u)\;dW_u.
$$

\smallskip

This kind of relation also holds when one considers functions defined on an interval. Indeed, for some fixed $-\infty<a<b<\infty$, and $H<1/2$, relation (\ref{deriv-frac-int}) yields
$$
\int_a^b\phi_u\;dB_u=\int_{\mathbb{R}}\phi_u\Ind_{[a,b]}(u)\;dB_u
=\int_{\mathbb{R}}(\mathcal{K}[\phi\Ind_{[a,b]}])(u)\;dW_u
=\int_{\mathbb{R}}(\mathcal{K}^{[a,b]}\phi)(u)\;dW_u,
$$
where the operator $\mathcal{K}^{[a,b]}$ is defined by:
$$
(\mathcal{K}^{[a,b]}f)(u)\equiv
\frac{\Gamma\left(H+1/2\right)}{C_1(H)}
\left(\mathcal{D}_{-b}^{1/2-H}f\right)(u),
\quad\mbox{for}\quad a<u<b,
$$
with $C_1(H)$ defined by (\ref{eq:61a}). In case of an interval
$[a,b]$, it should also be mentioned that an important subspace of
integrable processes is the following: let $\mathcal{E}^{[a,b]}$
be the set of step-functions on $[a,b]$ with values in
$\mathbb{R}$. As in \cite[Subsection 5.2.3]{Nu06}, we consider on
this space the semi-norm
$$
\|\varphi\|_{\mathcal{H}_K([a,b])}^2=\int_a^b
\frac{\varphi_u^2}{(b-u)^{1-2H}}du+\int_a^b\left(\int_u^b\frac{|\varphi_r-\varphi_u|}{(r-u)^{3/2-H}}dr\right)^2du.
$$
Let $\mathcal{H}_K([a,b])$ be the Hilbert space defined as the
closure of $\mathcal{E}^{[a,b]}$ with respect to the previous
semi-norm. Then the space $\mathcal{H}_K([a,b])$  is continuously
included in $\mathcal{H}$, and if $\phi\in
\mathbb{D}^{1,2}(\mathcal{H}_K([a,b]))$, then $\phi\in {\rm
Dom}(I)$.

\subsubsection{Generalized stochastic integrals}
The stochastic integrals we shall use in order to define our
doubly delayed L\'evy area and volume are defined,
in a natural way, by Russo-Vallois' symmetric approximations, that
is, for a given process $\phi$:
$$
\int_a^b \phi_w \, d^\circ B^i_w =L^2-\lim_{\varepsilon\rightarrow
0} \frac{1}{2\varepsilon} \int_a^b \phi_w \,
\big(B^i_{w+\varepsilon}-B^i_{w-\varepsilon}\big)dw,
$$
provided the limit exists. It is well known that the Russo-Vallois symmetric integral
coincides with Young's integral for $H>1/2$, and with the classical Stratonovich integral in the Brownian case $H=1/2$. Since these two cases are not very demanding from a technical point of view, we will focus our efforts on the case $1/4<H<1/2$. This being said, for $v_1\in[-r,r]$, $v_2\in[0,r]$, such that
$v_1+v_2\geq 0$, we will try to define the increments $\mathbf{B^2}$ and  $\mathbf{B^3}$ as
\begin{align}
\label{del:terms}
&\mathbf{B}^\mathbf{2}_{st}(v_1,v_2)=\int_{s-v_2}^{t-v_2} d^\circ
B_u \otimes \int_{s-v_2-v_1}^{u-v_1}d^\circ B_\tau, \text{ i.e. }
(\mathbf{B}^\mathbf{2}_{st}(v_1,v_2))^{ij}=\int_{s-v_2}^{t-v_2}
d^\circ B^j_u \int_{s-v_2-v_1}^{u-v_1}d^\circ
B^i_\tau \notag\\
&\mathbf{B}^\mathbf{3}_{st}(v_1,v_2)=\int_s^t d^\circ B_w
\otimes\int_{s-v_2}^{w-v_2} d^\circ B_u \otimes
\int_{s-v_2-v_1}^{u-v_1}d^\circ
B_\tau, \notag\\
&\hspace{4cm}\text{ i.e. }(\mathbf{B}^\mathbf{3}_{st}(v_1,v_2))^{ijk}=\int_s^t
d^\circ B^k_w\int_{s-v_2}^{w-v_2} d^\circ B^j_u
\int_{s-v_2-v_1}^{u-v_1}d^\circ B^i_\tau,
\end{align}
for all $i,j,k\in\{1,\ldots,d\}$, $0\leq s<t\leq T<\infty$.

\smallskip

Interestingly enough, one can establish the existence of symmetric integrals thanks to some Malliavin calculus criterions:

\begin{prop}
\label{integral-rep} Let $\phi$ be a stochastic process such that
$\phi\Ind_{[a,b]}\in\mathbb{D}^{1,2}(\mathcal{H}_K([a,b]))$, for
all $-\infty<a<b<\infty$. Suppose also that
$$
{\rm Tr}_{[a,b]}D\phi:=L^2-\lim_{\varepsilon\rightarrow
0}\frac{1}{2\varepsilon}\int_a^{b} \langle
D\phi_u,\Ind_{[u-\varepsilon,u+\varepsilon]}\rangle_\mathcal{H}du
$$
is an almost surely finite random variable. Then $\int_a^b \phi_u d^\circ B^i_u$ exists, and verifies
$$
\int_a^b \phi_u d^\circ B^i_u=I(\phi\Ind_{[a,b]})+{\rm
Tr}_{[a,b]}D\phi.
$$
\end{prop}
Furthermore, the following algebraic relation is trivially satisfied for this kind of integrals:
\begin{lem}\label{lem:5.2}
Let $\alpha=\{\alpha_w,\;w \in
[a,b]\}$ be a stochastic process such that its symmetric
Russo-Vallois integral with respect to a 1-dimensional fractional
Brownian motion $B$ exists, and let $F$ be a random variable. Then
$F\alpha$ is integrable with respect to $B$ in the Russo-Vallois
symmetric integral sense and
$\int_a^b F\alpha_w \, d^\circ B_w=F\int_a^b \alpha_w \, d^\circ B_w$.
\end{lem}

We are now ready to show the existence of delayed areas and
volumes with respect to fBm.

\subsection{Delayed Lévy areas}
Before we turn to statements involving increments as functions of two parameters, let us deal first with fixed times $s,t$:
\begin{prop}\label{prop:5.3}
Let $B$ be a $d$-dimensional fractional Brownian motion, with
Hurst parameter $H>1/4$. Then, for $s,t\in\ott$, $v_1\in[-r,r]$,
$v_2\in[0,r]$, such that $v_1+v_2\geq 0$, the doubly delayed
L\'evy area, denoted by $\mathbf{B}_{st}^\mathbf{2}(v_1,v_2)$ and
defined by (\ref{del:terms}), is well defined. In addition, we
have $\mathbb{E}[|\mathbf{B}_{st}^\mathbf{2}(v_1,v_2)|^2]\le c
|t-s|^{4H}$ for a strictly positive constant $c=c_{H,v_1,T}$
independent of $v_2$, exhibiting the following discontinuity phenomenon: we have $\lim_{v_1\to 0}c_{H,v_1,T}=\infty$, but $c_{H,0,T}$ is finnite.
\end{prop}

\begin{rem}
The discontinuity result on $c_{H,v_1,T}$ alluded to above is not a surprise, and had already been observed in \cite{NNT07}.
\end{rem}

\begin{proof}
As mentioned before, the case $H\ge 1/2$ is rather easy to handle, and we thus focus on $1/4<H<1/2$. It should also be mentioned that Lévy areas can be constructed in a similar way to \cite{NNT07}, though an extra attention has to be paid in order to treat irregular cases, when $H$ approaches $1/4$. As a last preliminary remark, observe that, due to the stationarity property of the fBm
we shall work without loss of generality on the interval $[0,t-s]$
instead of $[s-v_2,t-v_2]$ in the sequel, that is,
$\mathbf{B}^\mathbf{2}_{st}(v_1,v_2)$ behaves as
$\mathbf{B}^\mathbf{2}_{0,t-s}(v_1)=\mathbf{B}^\mathbf{2}_{0,t-s}(v_1,0)$.

\smallskip

{\noindent\it 1) Case} $i=j$ {\it and} $v_1\ge 0$. Consider the process
$\phi=(B^i_{\cdot-v_1}-B^i_{-v_1})\Ind_{[0,t-s]}(\cdot)$. When
$v_1\geq 0$, the arguments in \cite[Proposition 5.2]{NNT07} for
$1/3<H<1/2$ also hold for $1/4<H\leq 1/3$. Thus
\begin{equation*}
(\mathbf{B}^\mathbf{2}_{0,t-s}(v_1))^{ii}=I^{B^i}(\phi)+{\rm
Tr}_{[0,t-s]}D^{B^i}\phi,
\end{equation*}
where $I^{B^i}(\phi)$ denotes the divergence integral of $\phi$
with respect to $B^i$ and
$$
{\rm Tr}_{[0,t-s]}D^{B^i}\phi=\left\{
\begin{array}{cl}
\frac{1}{2}(t-s)^{2H}, &\text{if }v_1=0,    \\
-Hv_1^{2H-1}(t-s)+\frac{1}{2}\big((t-s+v_1)^{2H}-v_1^{2H}\big),
&\text{if } v_1>0.
\end{array}
\right.
$$
In addition, one can also prove, as in \cite{NNT07}, that
$$
\mathbb{E}\big|(\mathbf{B}^\mathbf{2}_{0,t-s}(v_1))^{ii}\big|^2\leq
c_{H,v_1}|t-s|^{4H},
$$
for any $v_1\in(0,r]$, where $\lim_{v_1\to 0}c_{H,v_1}=\infty$. On the other hand, the computations above also show that $\mathbb{E}|(\mathbf{B}^\mathbf{2}_{0,t-s}(v_1,0))^{ii}|^2\leq c_{H,v_1} |t-s|^{4H}$.

\smallskip

{\noindent\it 2) Case} $i=j$ {\it and} $v_1< 0$.
When $v_1<0$, we will show that
\begin{equation}\label{eq:70a}
(\mathbf{B}^\mathbf{2}_{0,t-s}(v_1))^{ii}=I^{B^i}(\phi)+{\rm
Tr}_{[0,t-s]}D^{B^i}\phi,
\end{equation}
where now
\begin{equation}
\label{delay-positive} {\rm
Tr}_{[0,t-s]}D^{B^i}\phi=H(-v_1)^{2H-1}(t-s)+\frac{1}{2}\big(|t-s+v_1|^{2H}-(-v_1)^{2H}\big).
\end{equation}
Indeed, notice that
$D_r^{B^i}\phi_u=\Ind_{[-v_1,u-v_1]}(r)\Ind_{[0,t-s]}(u)$ and furthermore, for
$u\in[0,t-s]$ and $\varepsilon \in[0,-v_1]$, one can write
\begin{align*}
&\langle \Ind_{[-v_1,u-v_1]},\Ind_{[u-\varepsilon,u+\varepsilon]}
\rangle_{\mathcal{H}}\\
&=\frac{1}{2}\big(|-v_1+\varepsilon|^{2H}-|-v_1-\varepsilon|^{2H}+|-v_1-u-\varepsilon|^{2H}-|-v_1-u+\varepsilon|^{2H}\big)\\
&=\frac{1}{2}\big((-v_1+\varepsilon)^{2H}-(-v_1-\varepsilon)^{2H}+|-v_1-u-\varepsilon|^{2H}-|-v_1-u+\varepsilon|^{2H}\big).
\end{align*}
Performing now a Taylor expansion in a neighbourhood of
$\varepsilon=0$, we get
\begin{equation*}
(-v_1+\varepsilon)^{2H}-(-v_1-\varepsilon)^{2H}=4H(-v_1)^{2H-1}\varepsilon+o\left(
\varepsilon ^{2}\right).
\end{equation*}
Thus, applying the dominated convergence theorem (details
are left to the reader) we obtain
\begin{equation}\label{eq:71}
\lim_{\ep\to 0}
\int_0^{t-s}\frac{1}{4\varepsilon}\big((-v_1+\varepsilon)^{2H}-(-v_1-\varepsilon)^{2H}\big)du
= H(-v_1)^{2H-1}(t-s).
\end{equation}
Along the same lines, by separating the cases $-v_1\geq t-s$, $0<u<-v_1<t-s$ and
$-v_1\leq u<t-s$, it can also be proved that
\begin{equation}\label{eq:72}
\lim_{\ep\to 0}\int_0^{t-s}\frac{1}{4\varepsilon} \big
(|-v_1-u-\varepsilon|^{2H}-|-v_1-u+\varepsilon|^{2H} \big)du =
\frac{1}{2}\big(|t-s+v_1|^{2H}-(-v_1)^{2H}\big).
\end{equation}
We now obtain (\ref{delay-positive}) by putting together (\ref{eq:71}) and (\ref{eq:72}).

\smallskip

Let us bound now ${\rm Tr}_{[0,t-s]}D^{B^i}\phi$ from expression (\ref{delay-positive}):
in the case $-v_1\geq t-s$, invoking the fact that, for $0<p<1$ and $a\geq
b>0$, the inequality $a^p-b^p\leq (a-b)^p$ holds true, we obtain
\begin{align*}
\big|{\rm
Tr}_{[0,t-s]}D^{B^i}\phi\big|&=H(-v_1)^{2H-1}(t-s)+\frac{1}{2}\big((-v_1)^{2H}-(-v_1-(t-s))^{2H}\big)\\
&\leq
H(t-s)^{2H}+\frac{1}{2}\big((-v_1)^{2H}-((-v_1)^{2H}-(t-s)^{2H})\big)\leq(t-s)^{2H},
\end{align*}
and in the case $-v_1<t-s$, we also have
\begin{align*}
\big|{\rm
Tr}_{[0,t-s]}D^{B^i}\phi\big|&=H(-v_1)^{2H-1}(t-s)+\frac{1}{2}\big|(t-s+v_1)^{2H}-(-v_1)^{2H}\big|\\
&\leq
H(-v_1)^{2H-1}T^{1-2H}(t-s)^{2H}+\frac{1}{2}\big((t-s)^{2H}+(-v_1)^{2H}+(-v_1)^{2H}\big)\\
&\leq \left(H(-v_1)^{2H-1}T^{1-2H}+\frac32\right)(t-s)^{2H}.
\end{align*}
Thus, we have found
\begin{align}
\label{del:bound_trace} \big|{\rm
Tr}_{[0,t-s]}D^{B^i}\phi\big|\leq
\left(H(-v_1)^{2H-1}T^{1-2H}+\frac32\right)(t-s)^{2H},
\quad\mbox{for all}\quad v_1\in[-r,0).
\end{align}

\smallskip

We proceed now to bound the term $I^{B^i}(\phi)$ in (\ref{eq:70a}): owing to (\ref{deriv-int}), we have
\begin{align}
\label{Malliavin-derivative}
D_r^{B^i}I^{B^i}(\phi)&=(B_{r-v_1}^i-B_{-v_1}^i)\Ind_{[0,t-s]}(r)+I^{B^i}\big(\Ind_{[-v_1,\cdot-v_1]}\Ind_{[0,t-s]}(\cdot)\big)\notag\\
&=(B_{r-v_1}^i-B_{-v_1}^i)\Ind_{[0,t-s]}(r)+I^{B^i}\big(\Ind_{[v_1+r,t-s]}(\cdot)\big)\Ind_{[-v_1,t-s-v_1]}(r)\notag\\
&=(B_{r-v_1}^i-B_{-v_1}^i)\Ind_{[0,t-s]}(r)+(B^i_{t-s}-B^{i}_{v_1+r})\Ind_{[-v_1,t-s-v_1]}(r).
\end{align}
Hence, thanks to (\ref{Malliavin-derivative}) and using the same arguments as in
the proof of \cite[Propositon 5.2]{NNT07}, we obtain
\begin{equation}
\label{del:2-term} \mathbb{E}|I^{B^i}(\phi)|^2\leq c_H|t-s|^{4H},
\end{equation}
with a constant $c_H>0$ independent of $v_1$.

\smallskip

Finally, (\ref{del:bound_trace}) and (\ref{del:2-term}) imply
$\mathbb{E}|(\mathbf{B}^\mathbf{2}_{0,t-s}(v_1))^{ii}|^2\leq
c_{H,v_1}|t-s|^{4H}$ for any $v_1\in[-r,0)$, and thus, according to our stationarity argument:
\begin{equation}\label{eq:77}
\mathbb{E}\big|(\mathbf{B}^\mathbf{2}_{st}(v_1,v_2))^{ii}\big|^2\leq
c_{H,v_1}|t-s|^{4H},
\end{equation}
for any $v_1\in[-r,0)$ and $v_2\in[0,r]$.

\smallskip

{\noindent\it 3) Case $i\neq j$.} This case can be treated similarly to \cite[Proposition 5.2]{NNT07}, and yields the same kind of inequality as in equation (\ref{eq:77}).

\smallskip

Our claim $\mathbb{E}[|\mathbf{B}_{st}^\mathbf{2}(v_1,v_2)|^2]\le
c |t-s|^{4H}$ now stems easily from the inequalities we have
obtained for the 3 cases $i=j$ and $v_1\ge 0$, $i=j$
and $v_1< 0$, and $i\ne j$.

\end{proof}

We can go one step further, and state a result concerning $\mathbf{B^{2}}$ as an increment.
\begin{prop}\label{prop:5.4}
Let $\mathbf{B^{2}}$ be the increment defined at Proposition \ref{prop:5.3}. Then $\mathbf{B^{2}}$ satisfies Hypothesis \ref{del:h1.1} and \ref{del:h1.2}.
\end{prop}

\begin{proof}

First, we have to ensure the almost sure existence of
$\mathbf{B}^\mathbf{2}_{st}(v_1,v_2)$ for  all $s,t\in\ott$. This
can be done by noticing that $\mathbf{B}^\mathbf{2}_{st}(v_1,v_2)$
is a random variable in the second chaos of the fractional
Brownian motion $B$, on which all $L^p$-norms are equivalent for
$p>1$. Hence we can write:
\begin{equation}
\label{del:bound-2-term}
\mathbb{E}\big|(\mathbf{B}^\mathbf{2}_{st}(v_1,v_2))^{ij}\big|^p\leq
c_{H,v_1,p}|t-s|^{2pH},
\end{equation}
for any $i,j\in\{1,\ldots,d\}$ and $p\geq 2$. With the same kind of calculations, one can also obtain the inequality
$$
\mathbb{E}\big|(\mathbf{B}^\mathbf{2}_{s_2 t_2}(v_1,v_2))^{ij}
-(\mathbf{B}^\mathbf{2}_{s_1 t_1}(v_1,v_2))^{ij}\big|^p\leq
c_{H,v_1,p}\lp |t_2-t_1|^{pH} + |s_2-s_1|^{pH} \rp.
$$
Then, a standard application of Kolmogorov's criterion yields the almost sure definition of the whole family $\{\mathbf{B}^\mathbf{2}_{st}(v_1,v_2); \, s,t\in\ott\}$, and its continuity as a function of $s$ and $t$.

\smallskip

Moreover, a direct application of Lemma \ref{lem:5.2} gives\begin{equation}
\label{del:delta-2-term} \delta
\mathbf{B}^\mathbf{2}(v_1,v_2)=\delta (B(v_2+v_1)) \otimes \delta
(B(v_2)),
\end{equation}
and Fubini's theorem for Stratonovich integrals with respect to $B$ also yield easily Hypothesis \ref{del:h1.2}. Finally, it is readily checked that  $\mathbf{B}^\mathbf{2}(v_1,v_2)\in
\mathcal{C}_2^{2\gamma}(\mathbb{R}^{d,d})$  for any
$1/4<\gamma<H$, $v_1\in[-r,r]$ (separating the case $v_1=0$) and $v_2\in[0,r]$. Indeed, it is sufficient to apply
Corollary 4 in \cite{Gu04} (see also inequality (90) in
\cite{NNT07}), having in mind the bound
(\ref{del:bound-2-term}) and expression
(\ref{del:delta-2-term}).

\end{proof}

\subsection{Delayed volumes}
We study now the term $\mathbf{B}^\mathbf{3}(v_1,v_2)$, starting from a similar statement as in Proposition \ref{prop:5.3}:
\begin{prop}\label{prop:5.5}
Let $B$ be a $d$-dimensional fractional Brownian motion, with
Hurst parameter $H>1/4$. Then, for $s,t\in\ott$, $v_1\in[-r,r]$,
$v_2\in[0,r]$, such that $v_1+v_2\geq 0$, the doubly delayed
volume, denoted by $\mathbf{B}^\mathbf{3}(v_1,v_2)$ and defined by
(\ref{del:terms}), is well defined. In addition, we have
$\mathbb{E}[|\mathbf{B}^\mathbf{3}(v_1,v_2)|^2]\le c |t-s|^{6H}$
for a strictly positive constant $c=c_{H,T,v_1,v_2}$ such that it
goes to $\infty$ if $v_1\to 0$ or $v_2\to 0$, but is also well
defined if $v_1=v_2=0$.
\end{prop}

\begin{proof}
Here again, we focus on the case $1/4<H<1/2$, and due tho the stationarity property of
the fBm, we shall work without loss of generality on the interval
$[0,t-s]$ instead of $[s,t]$ in the sequel. For notational sake, we will also set $\tau=t-s$ in the remainder of the proof.

\smallskip

{\noindent\it 1) Case} $i=j=k$. Consider the process
$\psi=(\mathbf{B}^\mathbf{2}_{0,\cdot}(v_1,v_2)^{ii})\Ind_{[0,\tau
]}(\cdot)$. We will define
$(\mathbf{B}^{\mathbf{3}}(v_1,v_2))^{iii}$ as $\int_0^{\tau
}\psi_u d^\circ B_u^i$, which amounts to show that
$\psi\in\mathbb{D}^{1,2}(\mathcal{H}_K([0,T]))$ and to compute the
trace of the process $\psi$.

\smallskip

With this aim in mind, let us first compute the Malliavin derivative of $\psi$: it is easily seen that
\begin{align}
\label{del:3.1-term-ijk}
&D_r^{B^i}\psi_u  \\
&=(B^i_{r-v_1}-B^i_{-v_2-v_1})\Ind_{[-v_2,u-v_2]}(r)\Ind_{[0,\tau ]}(u)
+I^{B^i}\big(\Ind_{[-v_2-v_1,\cdot-v_1]}(r)\Ind_{[-v_2,u-v_2]}(\cdot)\big)\Ind_{[0,\tau ]}(u)\notag\\
&=(B^i_{r-v_1}-B^i_{-v_2-v_1})\Ind_{[-v_2,u-v_2]}(r)\Ind_{[0,\tau ]}(u)
+(B^i_{u-v_2}-B^i_{r+v_1})\Ind_{[-v_2-v_1,u-v_2-v_1]}(r)\Ind_{[0,\tau ]}(u). \notag
\end{align}
From this identity, one can check that $\psi\in\mathbb{D}^{1,2}(\mathcal{H}_K([0,T]))$. We will now evaluate $\int_0^{\tau }\psi_u d^\circ B_u^i$ by separating the Skorokhod and the trace term in the symmetric integral.

\smallskip

\noindent
{\it (i) Evaluation of the trace term.}
We start by observing that $D^{B^i}\psi_u$ can also be written as:
\begin{align}
\label{del:3.2-term-ijk}
D_r^{B^i}\psi_u&=I^{B^i}\big(\Ind_{[\cdot+v_1,u-v_2]}(r)\Ind_{[-v_2-v_1,u-v_2-v_1]}(\cdot)\big)\Ind_{[0,\tau ]}(u)\notag\\
&\quad+I^{B^i}\big(\Ind_{[-v_2-v_1,\cdot-v_1]}(r)\Ind_{[-v_2,u-v_2]}(\cdot)\big)\Ind_{[0,\tau ]}(u).
\end{align}
Apply then Fubini's Theorem in order to get
\begin{align}
&\int_0^{\tau }\big\langle D^{B^i}\psi_u,
\Ind_{[u-\varepsilon,u+\varepsilon]}\big\rangle_{\mathcal{H}}du \notag\\
&=\int_{-v_2-v_1}^{\tau -v_2-v_1}\left(\int_{w+v_2+v_1}^{\tau }\langle
\Ind_{[w+v_1,u-v_2]},
\Ind_{[u-\varepsilon,u+\varepsilon]}\rangle_{\mathcal{H}}\,du\right)dB_w^i \label{eq:82}\\
&\quad+\int_{-v_2}^{\tau -v_2}\left(\int_{w+v_2}^{\tau }\langle
\Ind_{[-v_2-v_1,w-v_1]},
\Ind_{[u-\varepsilon,u+\varepsilon]}\rangle_{\mathcal{H}}\,du\right)dB_w^i,
\label{eq:83}
\end{align}
where the last two integrals have to be interpreted in the Wiener sense, and are well-defined according to the criterions in \cite{PT00}.

\smallskip

Let us evaluate the scalar product in (\ref{eq:82}):  for a fixed $v_2>0$, $u\in [w+v_2+v_1,\tau ]$,
$w\in[-v_2-v_1,\tau -v_2-v_1]$ and $\varepsilon\in[0,v_2]$, we can
write
\begin{align*}
&\langle \Ind_{[w+v_1,u-v_2]},
\Ind_{[u-\varepsilon,u+\varepsilon]}\rangle_{\mathcal{H}}\\
&=\frac{1}{2}\big(|-v_2+\varepsilon|^{2H}-|-v_2-\varepsilon|^{2H}
+|w+v_1-u-\varepsilon|^{2H}-|w+v_1-u+\varepsilon|^{2H}\big)\\
&=\frac{1}{2}\big((v_2-\varepsilon)^{2H}-(v_2+\varepsilon)^{2H}+(u-w-v_1+\varepsilon)^{2H}-(u-w-v_1-\varepsilon)^{2H}\big)\\
&=2H\big(-v_2^{2H-1}+(u-w-v_1)^{2H-1}\big)\varepsilon+o(\varepsilon^2).
\end{align*}
If $v_2=0$, one can prove similarly that for $\varepsilon$ small enough,
\begin{align*}
&\langle \Ind_{[w+v_1,u-v_2]},
\Ind_{[u-\varepsilon,u+\varepsilon]}\rangle_{\mathcal{H}}=2H(u-w-v_1)^{2H-1}\varepsilon+o(\varepsilon^2).
\end{align*}
This yields easily the relation
\begin{align*}
&\lim_{\ep\to 0}\frac{1}{2\varepsilon}\langle \Ind_{[w+v_1,u-v_2]},
\Ind_{[u-\varepsilon,u+\varepsilon]}\rangle_{\mathcal{H}} =
\left\{\begin{array}{cl}
H\big(-v_2^{2H-1}+(u-w-v_1)^{2H-1}\big)&\text{ if }v_2>0,    \\
H(u-w-v_1)^{2H-1}&\text{ if }v_2=0.
\end{array}
\right.
\end{align*}
The same kind of elementary arguments work for the scalar product
in expression (\ref{eq:83}), and one obtains:
$$
\lim_{\ep\to 0}\frac{1}{2\varepsilon}\langle
\Ind_{[-v_2-v_1,w-v_1]},
\Ind_{[u-\varepsilon,u+\varepsilon]}\rangle_{\mathcal{H}} =
H\big(-(u-w+v_1)^{2H-1}+(v_2+v_1+u)^{2H-1}\big).
$$
Thus, by an application of the dominated convergence theorem (whose details are left to the reader) we get, for a fixed $v_2>0$,
\begin{align}\label{eq:84}
&{\rm Tr}_{[0,\tau ]}D^{B^i}\psi\\
&=\int_{-v_2-v_1}^{\tau -v_2-v_1}\Big(-Hv_2^{2H-1}(\tau -w-v_2-v_1)+\frac{1}{2}[(\tau -w-v_1)^{2H}-v_2^{2H}]\Big)dB^i_w\notag\\
&+\frac{1}{2}\int_{-v_2}^{\tau -v_2}\Big((v_2+v_1)^{2H}-(\tau
-w+v_1)^{2H} +(\tau +v_2+v_1)^{2H}-(2v_2+v_1+w)^{2H}\Big)dB^i_w,
\notag
\end{align}
and for $v_2=0$, we end up with:
\begin{multline*}
{\rm Tr}_{[0,\tau ]}D^{B^i}\psi
=\frac{1}{2}\int_{-v_1}^{\tau -v_1}(\tau -w-v_1)^{2H}dB^i_w\\
+\frac{1}{2}\int_{0}^{\tau }\Big(v_1^{2H}-(\tau -w+v_1)^{2H}+(\tau
+v_1)^{2H}-(v_1+w)^{2H}\Big)dB^i_w.
\end{multline*}

\smallskip

For the remainder of the paper, the relation $a\lesssim b$ stands for $a\le C b$ with a universal constant $C$. Starting from equation (\ref{eq:84}), let us evaluate ${\rm Tr}_{[0,\tau ]}D^{B^i}\psi$ for $v_2>0$. Observe first that one can write
$
\mathbb{E}[|{\rm Tr}_{[0,\tau ]}D^{B^i}\psi|^2] \lesssim \sum_{l=1}^{4} J_l,
$
where $J_l$ can be decomposed itself as $J_l=\mathbb{E}[|\int_0^{\tau } F_l(w) dB^i_w|^2]$, with
\begin{align*}
&F_1(w)=(\tau -w),\quad F_2(w)= (\tau +v_2+v_1)^{2H}-(v_2+v_1+w)^{2H} \\
&F_3(w)=(v_2+v_1)^{2H}-(\tau -w+v_2+v_1)^{2H}, \quad
F_4(w)=(\tau -w+v_2)^{2H}-v_2^{2H}.
\end{align*}
Thus, thanks to relation (\ref{eq:65b}), we obtain:
$$
J_l=\|F_l\|_{\mathcal{H}([0,\tau ])}^2=
c_H\left\|\mathcal{D}_{-{\tau} }^{1/2-H}F_l\right\|_{L^2([0,\tau ])}^2.
$$
Furthermore, each $F_l$ is a power function, whose fractional
derivative $\mathcal{D}_{-{\tau} }^{1/2-H}F_l$ can be computed
explicitly. It is then easily shown that $\mathbb{E}|{\rm
Tr}_{[0,\tau ]}D^{B^i}\psi|^2 \le c_{H,v_2,T}{\tau} ^{6H}$, where
$c_{H,v_2,T}=c_{H,T}v_2^{2(2H-1)}+c_H$. Analogously, for $v_2=0$,
we get $\mathbb{E}|{\rm Tr}_{[0,\tau ]}D^{B^i}\psi|^2\leq
c_{H}{\tau} ^{6H}$.

\smallskip

\noindent {\it (ii) Evaluation of the Skorokhod term.} We shall
prove that $\mathbb{E}|I^{B^i}(\psi)|^2\leq c_{H,v_1,T}{\tau}
^{6H}$ and to this aim, let us decompose $\psi$ into its Skorokhod
and trace part. This gives $\mathbb{E}|I^{B^i}(\psi)|^2\leq
2\mathbb{E}|I^{B^i}(\psi_1)|^2+2\mathbb{E}|I^{B^i}(\psi_2)|^2$,
where
\begin{eqnarray*}
\psi_1(w)&=&\int_{-v_2}^{w-v_2}[B^i_{u-v_1}-B^i_{-v_2-v_1}]dB_u^i  \\
\psi_2(w)&=&{\rm Tr}_{[0,w]} D^{B^i}\phi,
\quad\mbox{with}\quad
\phi=(B^i_{\cdot-v_1}-B^i_{-v_2-v_1})\Ind_{[-v_2,w-v_2]}(\cdot).
\end{eqnarray*}
The proof that
\begin{align*}
\mathbb{E}|I^{B^i}(\psi_2)|^2\leq c_{H,v_1,T}{\tau} ^{6H},
\end{align*}
where $c_{H,v_1,T}\to \infty$ if $v_1\to 0$ but is also well
defined if $v_1=0$, can be obtained using the same arguments as
for Step (i), and we then concentrate on the Skorokhod term
$I^{B^i}(\psi_1)$.

\smallskip

To estimate $\mathbb{E}|I^{B^i}(\psi_1)|^2$, we use  first identity (\ref{duality}), which can be read here as  $\mathbb{E}|I^{B^i}(\psi_1)|^2$ $=\mathbb{E}[\langle \psi_1,\,  D^{B^i}I^{B^i}(\psi_1)\rangle_{\ch}]$. Taking into account relation (\ref{deriv-int}), the expression (\ref{del:3.1-term-ijk}) we have obtained for $D^{B^i}\psi_1$, and the isomorphism (\ref{eq:65b}), we end up with
\begin{equation}\label{eq:86}
\mathbb{E}\lc |I^{B^i}(\psi_1)|^2\rc \lesssim Q_1+Q_2+Q_3,
\end{equation}
where $Q_1,Q_2,Q_3$ are respectively defined by:
\begin{eqnarray*}
Q_1&=&\mathbb{E}\left\|\mathcal{D}_{-{\tau} }^{1/2-H}\psi_1\right\|_{L^2([0,\tau ])}^{2}  \\
Q_2&=&\mathbb{E}\left\|\mathcal{D}_{-(\tau -v_2)}^{1/2-H}\left(\int_{\cdot+v_2}^{\tau }[B^i_{\cdot-v_1}-B^i_{-v_2-v_1}]dB^i_w\right)\right\|_{L^2([-v_2,\tau -v_2])}^{2} \\
Q_3&=&\mathbb{E}\left\|\mathcal{D}_{-(\tau -v_2-v_1)}^{1/2-H}\left(\int_{\cdot+v_2+v_1}^{\tau }[B^i_{w-v_2}-B^i_{r+v_1}]dB^i_w\right)\right\|_{L^2([-v_2-v_1,\tau -v_2-v_1])}^{2}.
\end{eqnarray*}

\smallskip

We now estimate those 3 terms separately, starting with $Q_1$:
invoking the very definition (\ref{deriv-frac-int}) of the
fractional derivative $\mathcal{D}_{-{\tau} }^{1/2-H}$, it is
easily seen that $Q_1\lesssim A_1+A_2$, where
\begin{eqnarray*}
A_1&=&
\mathbb{E}\int_{0}^{\tau }\left(\int_{-v_2}^{r-v_2}[B^i_{u-v_1}-B^i_{-v_2-v_1}]dB^i_u\right)^2\frac{1}{(\tau -r)^{1-2H}}dr  \\
A_2&=&
\mathbb{E}\int_0^{\tau }\left(\int_r^{\tau }\frac{\int_{-v_2}^{w-v_2}[B^i_{u-v_1}-B^i_{-v_2-v_1}]dB^i_u-\int_{-v_2}^{r-v_2}[B^i_{u-v_1}-B^i_{-v_2-v_1}]dB^i_u}{(w-r)^{3/2-H}}dw\right)^2dr
\end{eqnarray*}
The term $A_1$ is easily bounded: according to Fubini's theorem and to our previous bounds on $\mathbf{B}^{\mathbf{2}}$, we have
\begin{multline*}
A_1=\int_{0}^{\tau }\mathbb{E}\left(\int_{-v_2}^{r-v_2}[B^i_{u-v_1}-B^i_{-v_2-v_1}]dB^i_u\right)^2\frac{1}{(\tau -r)^{1-2H}}dr  \\
\leq c_H\int_{0}^{\tau }r^{4H}\frac{1}{(\tau -r)^{1-2H}}dr
\leq
c_H{\tau} ^{4H}\int_{0}^{\tau }\frac{1}{(\tau -r)^{1-2H}}dr=\frac{c_H}{2H}{\tau} ^{6H}.
\end{multline*}
The term $A_2$ is a little longer to treat. However, by resorting
to the same kind of tools, one is able to prove that $A_2\leq c_H
{\tau} ^{6H}$, and gathering the estimates on $A_1$ and $A_2$, we
obtain $Q_1\leq c_H {\tau} ^{6H}$ as well. Finally, after some
tedious computations which will be spared to the reader for sake
of conciseness, we obtain the same kind of bound for $Q_2$ and
$Q_3$.

\smallskip

Now one has to reverse our decomposition process: putting together
our estimates on $Q_1,Q_2,Q_3$ and plugging them into
(\ref{eq:86}), we get $\mathbb{E}[ |I^{B^i}(\psi_1)|^2] \leq c_H
{\tau} ^{6H}$, with a constant $c_H > 0$ independent of $v_1,v_2$.
Finally, gathering the bounds on the Skorokhod and the trace term,
one obtains
$\mathbb{E}[|(\mathbf{B}^{\mathbf{3}}(v_1,v_2))^{iii}|^2]\le c
|t-s|^{6H}$.

\smallskip

{\noindent\it 2) Other cases.}
The previous arguments and computations can be simplified to obtain the desired
result for the case $i=k\neq j$ and $j=k\neq i$. The cases $i=j\neq k$ and $i\neq j\neq k$ can be treated by means of Wiener integrals estimations. This finishes the proof of our claim $\mathbb{E}[|\mathbf{B}^{\mathbf{3}}(v_1,v_2)|^2]\le c |t-s|^{6H}$.

\end{proof}

\smallskip

As in the case of delayed Lévy areas, and with exactly the same kind of arguments, one can push forward the analysis in order to deal with $\mathbf{B}^\mathbf{3}$ as an increment:
\begin{thm}
Let $\mathbf{B^{3}}$ be the increment defined at Proposition
\ref{prop:5.5}. Then $\mathbf{B^{3}}$ satisfies Hypothesis
\ref{del:h1.1}. Taking into account Proposition \ref{prop:5.4}, Theorem
\ref{del:t2.1} can thus be applied almost surely to the paths of the
$d$-dimensional fBm with Hurst parameter $H>1/4$.
\end{thm}

\medskip

\noindent {\bf Acknowledgment:} I. Torrecilla wishes to thank the IECN (Institut {\'E}lie Cartan Nancy) for its warm hospitality during a visit in 2008, which served to settle the basis of the current paper.

\end{document}